\documentclass[smallextended]{svjour3}       
\usepackage{graphicx}
\usepackage{grffile}
\usepackage[utf8]{inputenc} 
\usepackage[T1]{fontenc}    
\usepackage{hyperref}       
\usepackage{url}            
\usepackage{booktabs}       
\usepackage{amsfonts}       
\usepackage{nicefrac}       
\usepackage{microtype}      
\usepackage{cite}
\usepackage{amsmath,amssymb}
\usepackage{algorithm, algorithmic}
\usepackage{url}
\usepackage{pgf}
\usepackage{tikz}
\usetikzlibrary{arrows,automata}
\usepackage{multirow}
\usetikzlibrary{spy}
\usepackage{pgfplots}
\usepackage{enumitem}
\usepackage{mathtools}
\usepackage{subcaption}

\usepackage{lmodern}

\pgfplotsset{width=7cm,compat=1.8}
\newcommand{\R}{\mathcal{R}}
\DeclareMathOperator*{\argmin}{arg\,min}
\DeclareMathOperator{\prox}{prox}
\DeclareMathOperator{\id}{id}
\DeclareMathOperator{\spec}{spec}
\DeclareMathOperator{\PnP}{PnP}
\DeclareMathOperator{\range}{range}

\renewcommand{\epsilon}{\varepsilon}
\newcommand{\RegFunc}{g}
\newcommand{\NNparam}{\theta}
\newcommand{\ForwardOp}{A}
\newcommand{\Expect}{\mathbb{E}}
\newcommand{\Real}{\mathbb{R}}
\newcommand{\PClass}{\mathbb{P}}
\newcommand{\EncSpace}{\Xi}
\newcommand{\Encoder}{\mathcal{E}}
\newcommand{\Op}[1]{\operatorname{\mathcal{#1}}}

\newcommand{\ac}[1]{#1}
\newcommand{\ie}{i.e.,}

\newcommand{\Dictionary}{\mathbb{D}}
\newcommand{\oper}[1]{\operatorname{\mathcal{#1}}}
\newcommand{\AnalysisOp}{\oper{E}}
\newcommand{\SynthesisOp}{\AnalysisOp^{\ast}}

\newcommand{\EncCoeff}{\xi}
  
  \newcommand{\EncCoeffEst}{\hat{\EncCoeff}}

\newcommand{\NNLoss}{L}
\newcommand{\MLrecoparam}{\theta}

\newcommand{\RecSpace}{X}

\newcommand{\signal}{x}

\newcommand{\stsignal}{\mathbf{\signal}}

\newcommand{\DataSpace}{Y}

\newcommand{\data}{y}

\newcommand{\stdata}{\mathbf{\data}}

\newcommand{\RecOp}{\oper{R}}

\newcommand{\DataDiscrep}{f}
\newcommand{\RegParam}{\lambda}

\newcommand{\Lp}{L}
\newcommand{\SAset}[1]{ \{ {#1} \} } 

\begin{document}

\title{Convergent regularization in inverse problems and linear plug-and-play denoisers
}

\titlerunning{Convergent regularization and linear plug-and-play denoisers}

\author{Andreas Hauptmann \and Subhadip Mukherjee\and Carola-Bibiane Sch\"onlieb \and Ferdia Sherry   
}


\institute{AH is with the Research Unit of Mathematical Sciences, University of Oulu, Finland and also with the Department of Computer Science, University College London, UK. SM is with the	Dept. of Electronics and Electrical Communication Engineering, IIT Kharagpur, India. CBS and FS are
with the Dept. of Applied Mathematics and Theoretical Physics, University of Cambridge, UK.\\
Emails: Andreas.Hauptmann@oulu.fi, fs436@cam.ac.uk, smukherjee@ece.iitkgp.ac.in, cbs31@cam.ac.uk.          
}

\date{Received: date / Accepted: date}

\maketitle

\begin{abstract}
	Plug-and-play (PnP) denoising is a popular iterative framework for solving imaging inverse problems using off-the-shelf image denoisers. Their empirical success has motivated a line of research that seeks to understand the convergence of PnP iterates under various assumptions on the denoiser. While a significant amount of research has gone into establishing the convergence of the PnP iteration for different regularity conditions on the denoisers, not much is known about the asymptotic properties of the converged solution as the noise level in the measurement tends to zero, i.e., whether PnP methods are provably convergent regularization schemes under reasonable assumptions on the denoiser. This paper serves two purposes: first, we provide an overview of the classical regularization theory in inverse problems and survey a few notable recent data-driven methods that are provably convergent regularization schemes. We then continue to discuss PnP algorithms and their established convergence guarantees. Subsequently, we consider PnP algorithms with linear denoisers and propose a novel spectral filtering technique to control the strength of regularization arising from the denoiser. Further, by relating the implicit regularization of the denoiser to an explicit regularization functional, we rigorously show that PnP with linear denoisers leads to a convergent regularization scheme. More specifically, we prove that in the limit as the noise vanishes, the PnP reconstruction converges to the minimizer of a regularization potential subject to the solution satisfying the noiseless operator equation. The theoretical analysis is corroborated by numerical experiments for the classical inverse problem of tomographic image reconstruction. 
																									
\end{abstract}

\section{Introduction}
\label{intro}
Inverse problems deal with the estimation of an unknown model parameter $x^{*}\in X$ from its noisy and indirect measurement $y^{\delta}\in Y$ given by
\begin{equation}
	y^{\delta} = Ax^{*} + e.
	\label{eqn:obs_eq}
\end{equation}
We consider the case where $X$ and $Y$ are (potentially infinite dimensional) separable Hilbert spaces and $A:X\rightarrow Y$ is a bounded linear operator. $X$ and $Y$ are endowed with inner products $\langle\cdot, \cdot \rangle_X$ and $\langle\cdot, \cdot \rangle_Y$, inducing the norms $\Vert\cdot\Vert_X$ and $\Vert\cdot\Vert_Y$, respectively. The measurement noise level is bounded by $\delta$, i.e., $\Vert e\Vert_Y\leq \delta$. The clean measurement is denoted by $y^0$.  

The inverse problem in \eqref{eqn:obs_eq} is considered ill-posed in the sense of Hadamard, if either \textit{injectivity} or \textit{surjectivity} of the forward operator, or \textit{stability} of the solution map is violated. For instance, if $A$ is a compact operator with an infinite-dimensional range, then surjectivity and stability are not satisfied. This is, for example, the case for the ray transform operator that underlies many applications in medical imaging, such as computed tomography (CT) and positron emission tomography (PET) \cite{natterer1986mathematics,natterer2001mathematical}. The study of inverse problems usually assumes ill-posedness, as we will also do in the following.

In order to address ill-posedness, one needs to introduce a general concept for stable and unique solvability for an inverse problem of the form \eqref{eqn:obs_eq}. Due to the aforementioned ill-posedness, we can not guarantee to recover the true solution $x^*$ for all measurements and hence we first need the concept of a generalized solution. 
A common approach is to search for solutions that are closest to the measured data with respect to a suitable data discrepancy term $f:\DataSpace\times\DataSpace \to \Real_+$, such as the (squared) distance in the norm, i.e., $f(\ForwardOp\signal,\data^\delta)=\|\ForwardOp\signal-\data^\delta\|_Y^2$. Then we search for $\widetilde{x}\in X$ such that
\begin{equation}\label{eqn:dataDiscMini}
	f(\ForwardOp\widetilde{\signal},\data^\delta) \leq  f(\ForwardOp\signal,\data^\delta) \hspace{0.5 cm} \text{for all } x\in X.
\end{equation}
\eqref{eqn:dataDiscMini} implies that $\widetilde{x}$ is closest to the measured data with respect to $\DataDiscrep$, which deals with the violation of surjectivity by disregarding components of $y^\delta$ in the co-kernel of $\ForwardOp$. Furthermore, if $\ForwardOp$ has a non-trivial null space, then $\widetilde{x}$ is not unique. To obtain a unique solution, one can define the minimum norm solution as
\begin{equation}\label{eqn:dataDiscMini_minL2}
	x^\dagger=\arg \min_{x \in X}\{\|x\|_X : \ x  \text{ minimizes }f(\ForwardOp\signal,\data^\delta)\}.
\end{equation}
The element $x^\dagger$ can now be considered a desirable generalized solution to \eqref{eqn:obs_eq}.
When $f$ and $\|\cdot\|_X$ are given by the  squared $L^2$-norm, we call $x^\dagger$ the least-squares minimum-norm solution and can define a mapping $A^\dagger:Y\to X$, such that $x^\dagger=A^\dagger y^\delta$. In fact, the mapping $A^\dagger$ defines what is referred to as the Moore-Penrose pseudo-inverse. Unfortunately, if the operator $A$ is compact, then $A^\dagger$ will be unbounded and as such does not take care of the stability problem in the presence of noise in the data. This is where the concept of \textit{regularization} becomes important, as we will discuss next.

In order to deal with the stability issue, \emph{regularization theory} considers specifically designed solution maps. Such a solution map $\R({\cdot;\lambda}):Y\rightarrow X$, also called a \textit{reconstruction operator}, is expressed as a parametric map that produces an estimate of $x^*$ given $y^{\delta}$. Here, the parameter $\lambda$ depends on the noise level $\delta$, and we denote this explicitly by the mapping $\delta\rightarrow\lambda(\delta)$. In this paper, we are specifically interested in the notion of \textit{convergent regularization} which can be understood as \textit{convergence} of the reconstruction operator when the noise level $\delta$ tends to zero. More specifically, we want that when the noise level $\delta\to 0$, then $\lambda(\delta)\rightarrow \lambda_0\geq 0$, and the reconstruction operator $\R({y^{\delta};\lambda})$ converges to a generalized solution of the noiseless operator equation 
\begin{equation}
	Ax=y^{0}.
	\label{eq:noiseless_op_eqn}
\end{equation}
In the following, we will first review classical approaches to regularization in Section \ref{sec:Regularization_overview} and why inverse problems necessitate a generalized notion of solvability. We will then continue to discuss how this classical approach can be combined with modern data-driven methods. In particular, we will devote special attention to the so-called plug-and-play (PnP) approaches in Section \ref{sec:Reg_byPnP}, which have been shown to yield excellent empirical results for imaging inverse problems. These methods utilize state-of-the-art denoisers, model- or data-driven, to replace the proximal operators within iterative proximal splitting algorithms for solving an underlying variational minimization problem for reconstruction. We will subsequently provide an analysis in Section \ref{sec:controll_regPnP} of how PnP approaches can in fact provide a convergent regularization method in the classical sense, in particular, with linear denoisers. 



\section{Regularization for inverse problems and data-driven methods}
\label{sec:Regularization_overview}
Regularization theory has been a rich and successful field in inverse problems for several decades. The primary motivation is to formulate a well-posed and stable inversion procedure that converges provably to a solution of the noiseless operator equation \eqref{eq:noiseless_op_eqn}. The emergence of data-driven methods has given the field of inverse problems a new direction: by using large quantities of data we can significantly improve reconstruction results. However, the underlying question of a convergent regularization remains: does the obtained reconstruction solve the underlying operator equation? 

Indeed, there exist a few methods that are provably convergent regularization methods, we refer to \cite{mukherjee2023learned} for a survey. In the following, we will give a short overview of the regularization theory and existing data-driven approaches that are provably convergent regularization methods in this context.
\subsection{Classical regularization theory}
Stable solutions to inverse problems need a way to handle varying noise levels. For this purpose, the concept of regularization has proven highly useful. Roughly, regularization can be understood as a convergence requirement to a unique solution, e.g., the minimum norm  solution $x^\dagger$, where convergence depends on the noise level $\delta$. That is, formally we consider the previously discussed reconstruction operator $\R_\lambda := \R(\cdot,\lambda)$, which provides a parameterized family of continuous operators $\R_\lambda:Y\to X$. The parameter $\lambda$ depends on the noise level $\delta>0$, where $\| y^{\delta} - y^0 \| \leq \delta$ and $y^0 := A x^*$ denotes noise-free data. We say that the family of reconstruction operators is a convergent regularization method if there exists a parameter choice rule $\delta \mapsto \lambda(\delta, y^{\delta})$ such that reconstructions $x^\delta:=\R_{\lambda(\delta,y^\delta)}(y^\delta)$ converge to the solution $x^\dagger:=A^{\dagger} y^0$ given by the pseudo-inverse as noise vanishes, in the sense that
\begin{equation}\label{eqn:classicRegCondition}
	\limsup_{\delta \to 0} \,\bigl\Vert x^\delta
	-
	x^{\dagger}
	\bigr\Vert_{X}=0
	\quad\text{as}\quad
	\limsup_{\delta \to 0} \,
	\{\lambda(\delta,y^{\delta})\}
	= 0.
\end{equation}
In other words, we have point-wise convergence of the reconstruction operators to the pseudo-inverse, i.e., $\R_{\lambda(\delta, y^{\delta})}(y^\delta) \to A^\dagger y^0$ as $\delta\to 0$. We refer interested readers to \cite{engl1996regularization} for a detailed discussion. This is, of course, quite restrictive and only considers convergence to the least-squares minimum-norm solution. Nevertheless, this can already be used as a useful tool to design learned regularization methods, i.e., learned reconstruction approaches that formally satisfy the above convergence criteria, as we will discuss in the following.
\subsubsection{Direct regularization}
Motivated by the convergence to the pseudo-inverse solution, one can obtain a regularization method by mimicking the construction of the pseudo-inverse. In finite dimensions, this can be achieved by the singular value decomposition (SVD) $A=USV^{\top}$ of the forward operator. The pseudo-inverse can then be simply obtained by $A^\dagger = VS^\dagger U^{\top}$, where $S^\dagger$ is the transposed singular value matrix with inverted singular values. A regularization method is now obtained by filtering the singular values with a noise-dependent filter function, or a noise level-dependent truncation. 

Similarly, direct reconstruction methods that apply a regularized inverse of the forward operator can be shown to be convergent regularization methods. The most prominent example of such methods is the filtered back-projection (FBP) for X-ray CT, which is, in fact, still relevant in clinical practice. Here, the filtering operation removes high-frequency components in Fourier space to regularize the reconstructions. If the filtering is interpreted as a noise-dependent mollifier, one obtains the general class of approximate inverse \cite{schuster2007method} with convergence as noise vanishes. 

A popular approach in data-driven methods is to formulate a reconstruction operator as composition of a regularized reconstruction operator $\R_\lambda:\DataSpace\to\RecSpace$ with a data-driven component $C_{\theta}:\RecSpace\to\RecSpace$. That is, the reconstruction operator is parameterized as $\R_{(\theta,\lambda)} := C_{\theta}\circ \R_\lambda$, where the data driven component $C_{\theta}$ is designed to improve the reconstruction by removing noise or undersampling artefacts, usually given by a deep convolutional neural network (CNN) \cite{postprocessing_cnn,kang2017deep}. These approaches are also popularly referred to as \textit{post-processing methods}. 



Such one-step post-processing approaches are especially popular due to their simplicity, as $C_{\theta}$ can be efficiently trained when supervised pairs of high and low-quality reconstructions are available. Unfortunately, there are very few results on reconstruction guarantees for such methods. Specifically, the problem formulation as a composition of a regularized reconstruction followed by the data-driven component causes the reconstruction to often violate the so-called \textit{data-consistency criterion}. That is, even if the data-fidelity $f(\bigl(\ForwardOp \circ \R_\lambda\bigr)(\data^\delta),\data^\delta)$ is small, it does not necessarily imply a small value of $f\bigl(\bigl(\ForwardOp \circ {C}_{\theta}\circ \R_\lambda\bigr)(\data^\delta),\data^\delta\bigr)$ corresponding to the output of the post-processing network ${C}_{\theta}$. Thus, such schemes do not satisfy the convergence of the data fidelity and hence fail to be a convergent regularization strategy. 

Nevertheless, as proposed in \cite{Schwab_2019_null_space}, this approach can be reformulated by constructing the post-processing network as ${C}_{\theta} = \id+\left(\id-A^{\dagger}A\right){Q}_{\theta}$, where ${Q}_{\theta}$ is a Lipschitz-continuous Deep Neural Network (DNN) and $\id$ denotes the identity operator on $X$. Here, $\left(\id-A^{\dagger}A\right)$ is the projection operator onto the null-space of $A$ and hence the operator ${C}_{\theta}$ (referred to as null-space network) always satisfies $\bigl(\ForwardOp \circ {C}_{\theta}\circ \R_\lambda\bigr)(\data^\delta)=\bigl(\ForwardOp \circ \R_\lambda\bigr)(\data^\delta)$, ensuring that the output of ${C}_{\theta}$ explains the observed data. More importantly, the null-space network maintains the regularizing properties of the reconstruction method $\R_\lambda$ and hence provides a convergent regularization scheme \cite{Schwab_2019_null_space} in the sense of direct regularization. See \cite{boink2023data}  for a recent extension of null-space networks to non-linear inverse problems.

\subsection{Variational regularization}
The classical regularization theory, which defines convergent regularization by convergence to the pseudo-inverse solution as defined in \eqref{eqn:classicRegCondition} limits possible solutions. Therefore, one can consider more general variational approaches to inverse problems, which have been particularly popular due to their flexibility in incorporating prior knowledge and dealing with varying noise distributions. In the variational regularization framework, solutions are computed by minimizing a composite objective consisting of the data-consistency term and a regularization term. In particular, the solutions are given by
\begin{equation}
	\R({y^{\delta};\lambda})\in \underset{x\in X}{\argmin}\,f(Ax,y^{\delta})+\RegFunc_{\lambda}(x).
	\label{eqn:varrecon_eq}
\end{equation}
The loss functional $f:Y\times Y\rightarrow\mathbb{R}^{+}$ measures data fidelity and is not restricted anymore to be the squared $L^2$-norm. 
The regularization functional $g_{\lambda}:X\rightarrow\mathbb{R}$ encodes prior belief about the ground-truth $x^*$ and effectively restricts the null space of $A$. In the general case, $\lambda$ can be a parameter of the functional, and more commonly, a simple weighting parameter to balance between the two terms of the composite objective in \eqref{eqn:varrecon_eq} (i.e., $g_{\lambda}(x)=\lambda g(x)$, where $g$ is the regularizer). 
The choice of a suitable regularizer $g_\lambda$ is governed by the need to balance two important factors: desirable analytical features and the encoded prior belief. For instance, an analytically favorable choice is given by the squared $L^2$-norm, which, in combination with a squared $L^2$-norm for the data fidelity, provides a closed-form solution. Unfortunately, the obtained solutions corresponding to this choice of the regularizer will be smooth, which may not be suitable for many imaging applications. Consequently, more advanced sparsity-promoting priors have been favored, most commonly the $L^1$-norm for sparse signals and total variation (TV) for sparse gradients, i.e., piece-wise constant functions. These regularizers are non-differentiable and hence need more advanced non-smooth optimization techniques to compute a minimizer \cite{benning2018modern}, but they typically lead to a better reconstruction than the simple squared $L^2$-norm-based regularization. See Figure \ref{fig:CT_different_reg} for a comparison of a few handcrafted regularizers in the context of sparse-view CT reconstruction.  
\begin{figure}
	\centering
	\begin{subfigure}{0.35\textwidth}
		\includegraphics[width=\textwidth]{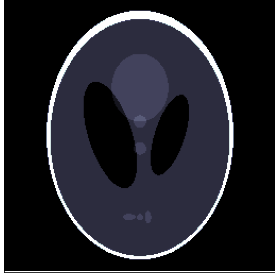}
		\caption{ground-truth}
	\end{subfigure}
	\begin{subfigure}{0.35\textwidth}
		\includegraphics[width=\textwidth]{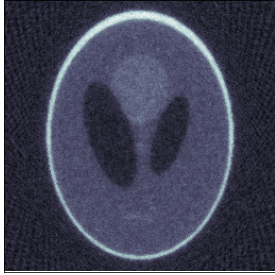}
		\caption{$g_{\lambda}(x)=\lambda\,\|x\|_2^2$}
	\end{subfigure}
	\begin{subfigure}{0.35\textwidth}
		\includegraphics[width=\textwidth]{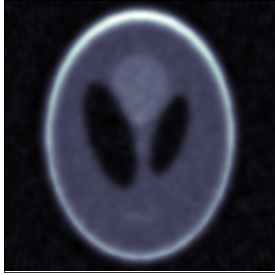}
		\caption{$g_{\lambda}(x)=\lambda\,\|\nabla x\|_2^2$}
	\end{subfigure}
	\begin{subfigure}{0.35\textwidth}
		\includegraphics[width=\textwidth]{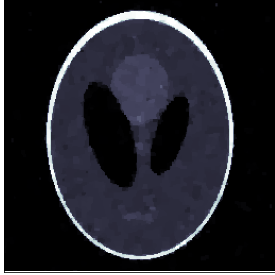}
		\caption{$g_{\lambda}(x)=\lambda\,\|\nabla x\|_1$}
	\end{subfigure}
	\caption{\small{Sparse-view CT reconstruction with different regularizers. The noisy sinogram is generated by first computing parallel-beam projections along 64 equally-spaced angular positions of the source, with 365 lines per position, and then by adding white Gaussian noise. The regularization parameter $\lambda$ is chosen to be $\lambda=0.1$ in (b) and $\lambda=10^{-4}$ in (c) and (d). This experiment demonstrates that the convergence of the regularization scheme is not particularly indicative of the quality of the reconstruction and underscores the need to learn a data-adaptive regularizer for enhancing the reconstruction quality.}}
	\label{fig:CT_different_reg}
\end{figure}


Notably, the role of the two terms in \eqref{eqn:varrecon_eq} is conceptually similar to the general formulation in \eqref{eqn:dataDiscMini_minL2} of a minimum-norm solution. Nevertheless, the variational formulation provides more flexibility and also necessitates a broader concept of regularization. This is because we can not always guarantee convergence to the minimum-norm solution, but we have to consider convergence with respect to the chosen regularization functional $\RegFunc_\lambda$ \cite{scherzer2009variational}.

We can then formulate a regularization method in the framework of variational regularization as follows: let us first denote $x_{\lambda} \in X$ to be a minimizer to the objective in \eqref{eqn:varrecon_eq} for a given $\lambda$ with data $y^{\delta} \in Y$ and noise level $\Vert y^\delta - y^0 \Vert_Y < \delta$. Similarly, as before, we assume that there is a corresponding parameter choice rule $\delta \mapsto \lambda(\delta,y^{\delta})$ such that $\lambda\rightarrow\lambda_0$ as $\delta\rightarrow 0$. The variational model defined by \eqref{eqn:varrecon_eq} is then said to \emph{converge to a $g$-minimizing solution} if $x_{\lambda(\delta,y^{\delta})} \to \hat{x}$ as $\delta \to 0$.
Here, $\hat{x} \in X$ solves the variational model that corresponds to \eqref{eqn:varrecon_eq} with noise-free data $y^0 \in Y$, i.e., 
\begin{equation}
	\hat{x} \in \argmin_{x \in X}\,\,
	g_{\lambda_0}(x) 
	\quad\text{subject to $y^0=A x^*$ and 
		where $\lambda_0 := \lim_{\delta \to 0}\lambda(\delta,y^{\delta})$.}
	\label{eq:R_min_def}
\end{equation}
The primary differences to the classical formulation here are, that the minimizer of the regularizing functional $g$ is not necessarily unique and the regularization parameter is not required to converge to 0. 

Let us remark to this end, that it is desirable to formulate a regularizer that has small values for the desired images, i.e., it penalizes undesired solutions but is also analytically or computationally tractable. 
It is important to note at this point that different regularizers $g$, which provide a convergent regularization, will still produce very different reconstruction results as illustrated in Figure \ref{fig:CT_different_reg}, as not all choices of $g$ are a good representation of the desired ground-truth image. Here, learned regularizers have proven very successful, as the data itself can now be used to represent the regularizer and hence naturally offer a good representation of the desired features. Depending on the choice of representation, analysis of the learned regularizer may become more involved. In the following, we will discuss several choices for learned data-driven regularizers and how these can be used within the realm of variational regularization. 


\subsection{Learning a regularizer}
The idea to learn a regularizer from data, rather than the classical approach of modeling it from first principles as outlined above, has appeared in the literature in various forms. We outline here a few such approaches, ranging from relatively older yet widely popular ideas like dictionary learning to the more recent approaches to learning regularizers using deep neural networks.
\subsubsection{Learning sparsity-promoting dictionaries}
We start with the concept of dictionary learning, which nicely illustrates how data can be used to learn a representation of the desired images. 
Here, we will use the concept of sparsity, which has long been important for modeling prior knowledge of solutions, to regularize inverse problems. Assuming that the reconstruction possesses a sparse representation in a given dictionary $\Dictionary$, one can develop sparse recovery strategies, associated computational approaches, and error estimates for the reconstruction. Instead of working with a given dictionary, the key idea is to \textit{learn} a dictionary either \textit{a-priori} or \textit{jointly} with the reconstruction. Notably, almost all work on dictionary learning in sparse models has been carried out in the context of denoising, i.e., with  $A=\id$. 


Learning the dictionary separately to solve the reconstruction problem is usually done using a sparsity assumption on the representation given by the dictionary. Let $\NNLoss_X : \RecSpace \times \RecSpace \to \Real$ be a given loss function (e.g.\ the $\Lp^2$- or $\Lp^1$-norm). Further, let $\signal_1, \ldots, \signal_N \in \RecSpace$ be the given unsupervised training data, $\Dictionary = \SAset{\phi_i}\subset\RecSpace $ a dictionary, and the synthesis operator $\SynthesisOp_{\Dictionary}:\EncSpace \to\RecSpace$ acting on the encoder space $\EncSpace$ given as $\SynthesisOp_{\Dictionary}(\EncCoeff) = \sum_i \EncCoeff_i \phi_i $ for $\EncCoeff\in\EncSpace$. One approach in dictionary learning is based on the idea of finding a dictionary that approximates the training data in the given loss function with the sparsest possible coefficients, by solving 
\begin{equation}
\label{eq:dictionary_learning-ver1}
	( \widehat{\Dictionary},\EncCoeffEst_{i} ) \in \argmin_{\EncCoeff_i\in\EncSpace,\Dictionary\subset\RecSpace}\sum_{i=1}^N \NNLoss_X (\signal_i,\SynthesisOp_{\Dictionary}(\EncCoeff_i) ), \text{\,such that\,}\| \EncCoeff_i \|_0\leq s, \text{\,for\,}1\leq i\leq N,
\end{equation}
where $s$ is a given sparsity level. Alternatively, one can formulate the optimality criterion by looking for a dictionary that maximizes the sparsity of the dictionary representation while enforcing a constraint on the precision in which the synthesis operation approximates, i.e., by seeking a dictionary such that $\NNLoss_X (\signal_i,\SynthesisOp_{\Dictionary}(\EncCoeff_i) ) \leq \epsilon$ for $i = 1,\ldots,N$, while maximizing sparsity.
A unified formulation is given by the unconstrained problem
\begin{equation}
	( \widehat{\Dictionary},\EncCoeffEst_{i} ) = \argmin_{\EncCoeff_i\in\EncSpace,\Dictionary\subset\RecSpace}\sum_{i=1}^N [ \NNLoss_X (\signal_i,\SynthesisOp_{\Dictionary}(\EncCoeff_i) ) +\theta \| \EncCoeff_i \|_0 ].
	\label{eq:dictionary_learning-unified0}
\end{equation}
Both \eqref{eq:dictionary_learning-ver1} and \eqref{eq:dictionary_learning-unified0} are posed in terms of
the $\Lp^0$-norm and are NP-hard problems. This suggests the use of \emph{convex relaxation}, by  replacing $ \| \EncCoeff_i \|_0 $ with $\| \EncCoeff_i \|_1$ in \eqref{eq:dictionary_learning-unified0}. This relaxation turns \eqref{eq:dictionary_learning-unified0} into a \textit{bi-convex problem} (convex in each variable when the others are kept fixed) subject to usual choices for $\NNLoss_X$, and one can apply alternating minimization approaches for obtaining an approximate solution. 
Seminal work on sparse dictionary learning includes the K-SVD approach \cite{Aharon:2006aa}, geometric multi-resolution analysis (GMRA) \cite{Allard:2012aa}, and online dictionary learning \cite{Mairal:2010aa}. See also \cite{Rubinstein:2010aa} and references therein for a thorough discussion on sparse dictionary learning approaches.  

While dictionary learning in the context of sparse coding has been very popular and successful, there are still several issues with it related to the locality of learned structures and the computational effort needed, for instance when sparse coding is performed over a large number of images or image patches.
Consequently, a computationally feasible approach is needed that introduces further structure and invariances in the dictionary (e.g.,\ shift-invariance), which makes each dictionary atom $\phi_i$ dependent on the whole image instead of just individual patches. In this context, 
convolutional dictionaries have been introduced. Here, the dictionary atoms are given by convolution kernels that act on signal features via convolution
and hence provide computationally feasible shift-invariant dictionaries, where the atoms depend on the entire signal/image.

Convolutional dictionary learning is formulated as follows. Given unsupervised training data $\signal_1,\ldots,\signal_N \in \RecSpace$ and a loss function $\NNLoss_X : \RecSpace \times \RecSpace\to \RecSpace$, one seeks to solve
\begin{equation}\label{eq:csclearning}
	\argmin_{\phi_i,\EncCoeff_{j,i}\in\RecSpace}\Biggl\{ 
	\sum_{j=1}^N \NNLoss_X\biggl(\signal_j,\sum_i \EncCoeff_{j,i} \ast \phi_i \biggr)  +\RegParam \sum_{j=1}^N \sum_i  \| \EncCoeff_{j,i} \|_1\Biggr\},
\end{equation}
where $\| \phi_i \|_2 = 1$. 
The above can be solved using an ADMM-type scheme, similar to what is done for the $L^2$-loss in \cite{Garcia-Cardona:2017aa}. There are various extensions of convolutional dictionary learning, for instance, multi-layer variants \cite{Sulam:2017aa}.

The dictionary can also be learned jointly with the reconstruction, by formulating a joint optimization problem. An example of such an approach is the adaptive dictionary-based statistical iterative reconstruction (ADSIR) \cite{Zhang:2016aa}, and its variants \cite{Xu:2012aa,Chun:2017aa}. A joint problem could be formulated as:
\begin{equation}\label{eq:sparselanddictionarylearn}
	\min_{\signal\in\RecSpace,\EncCoeff_i\in\EncSpace,\Dictionary} 
	\{ f(\ForwardOp\signal,\data )+ \RegFunc_{\lambda}(\signal,\EncCoeff_1,\ldots,\EncCoeff_N,\Dictionary) \} ,
\end{equation}
where
\begin{equation}
	\RegFunc_{\lambda}(\signal,\EncCoeff_1,\ldots,\EncCoeff_N,\Dictionary) :=\sum_{j=1}^N  \left[ \NNLoss_X (\signal_i,\SynthesisOp_{\Dictionary}(\EncCoeff_i) ) + \lambda \|\EncCoeff_j \|_p^p \right] ,
\end{equation}
while $\SynthesisOp_{\Dictionary}:\EncSpace\to\RecSpace$ being the synthesis operator associated with the dictionary $\Dictionary$. A convergent regularization could now be obtained under suitable conditions on $g_\lambda$ following the variational regularization framework.

Finally, a formulation in infinite dimensional spaces is studied in \cite{Chambolle:2018aa}, proposing a convex variational model for joint reconstruction and dictionary learning, that applies to inverse problems and allows to establish existence and stability guarantees for the reconstruction.
\subsubsection{Bilevel learning} 
Starting from variational regularization methods where the reconstruction operator $\RecOp_{\RegParam} \colon \DataSpace \to \RecSpace$ is defined as the solution map for \eqref{eqn:varrecon_eq}, one can formulate a generic setup for learning selected components of \eqref{eqn:varrecon_eq} utilizing supervised training data and a suitable loss function $\NNLoss_X \colon \RecSpace \times \RecSpace \to \Real$.
This setup can be tailored towards learning the regularization functional $\RegFunc_\lambda$ \cite{de2017bilevel,de2022bilevel}, the data fidelity term $\DataDiscrep$, or even an appropriate component in the forward operator $\ForwardOp$, e.g., in blind image deconvolution \cite{hintermuller2015bilevel}. Notably, the joint dictionary learning problem \eqref{eq:sparselanddictionarylearn} can also be formulated as a \textit{bilevel learning problem}.

First, we generalize the regularizer $\RegFunc_\lambda$ consisting of a single regularization parameter $\RegParam$ to a set of parameters $\NNparam$ (vector-valued). Subsequently, we define the reconstruction operator as 
\begin{equation}\label{eq:VarRegBiGeneral}
	\RecOp_{\NNparam}(\data)
	:= \argmin_{\signal \in \RecSpace} \{ 
	f(\ForwardOp\signal,\data ) + \RegFunc_{\NNparam}(\signal) 
	\}  
	\quad \text{for $\data \in\DataSpace$.}
\end{equation}
Given paired training data $(\signal_i , \data_i) \in  \RecSpace  \times \DataSpace$ that are i.i.d. samples of the $(\RecSpace  \times \DataSpace)$-valued random variable $(\stsignal,\stdata) \sim \pi_\mathrm{joint}$, we can formulate the following \emph{bi\-level learning} problem:
\begin{equation}\label{eq:bilevel}
	\begin{cases}   
		\widehat{\theta} \in \displaystyle{\argmin_{\theta}\,  \Expect_{(\stsignal,\stdata) \sim \pi_\mathrm{joint}} [ \NNLoss_X (\RecOp_{\theta}(\stdata),\stsignal ) ], \text{\,\,where}} &   
		\\
		\RecOp_{\theta}(\data)
		:= \displaystyle{\argmin_{\signal \in \RecSpace}} \{ 
		f(\ForwardOp(\signal),\data )
		+ \RegFunc_{\theta}(\signal) 
		\}.                                                                                                                                                                                 &   
	\end{cases}
\end{equation}
Note that $\widehat{\theta}$ here is a Bayes estimator. However, the true joint distribution $\pi_\mathrm{joint}$ is typically unknown and is replaced by its empirical counterpart given by the training data, in which case $\widehat{\theta}$ corresponds to \emph{empirical risk minimization}.

In the bilevel optimization literature, as in the optimization literature as a whole, there are two main and mostly distinct approaches. In the discrete approach that first discretizes the problem \eqref{eq:VarRegBiGeneral} and subsequently optimizes its parameters. In this way, optimality conditions and their well-posedness are derived in finite dimensions. Alternatively, $\RecOp$ and its parameter $\MLrecoparam$ in~\eqref{eq:bilevel} are optimized in the continuum (\ie\ appropriate infinite-dimen\-sional function spaces) and then discretized. It should be noted that the resulting problems present several difficulties due to the frequent non-smoothness of the lower-level problem (think of \ac{TV} regularization), which, in general, makes it impossible to verify Karush--Kuhn--Tucker constraint qualification conditions. This issue has led to the development of alternative analytical approaches in order to obtain first-order necessary optimality conditions \cite{DeLosReyes2011,hintermuller2014elliptic}.

\subsubsection{Adversarial regularization} Another notable alternative approach to include a learned regularization in the reconstruction process is to learn an explicit regularization term in \eqref{eqn:varrecon_eq} and solve the variational problem subsequently. One such option is to learn \textit{adversarial regularizers} as first proposed in \cite{ar_nips} and further developed in \cite{acr_arxiv}. Here, the construction of data-driven regularization is inspired by how discriminative networks (also referred to as \textit{critics}) are trained using modern Generative Adversarial Network (GAN) architectures. 

To train such an adversarial regularizer, we assume to have $\left\{x_i\right\}_{i=1}^{n_1}\in {X}$ and $\left\{y_i\right\}_{i=1}^{n_2}\in {Y}$, which are i.i.d. samples from the marginal distributions $\pi_x$ and $\pi_y$ of ground-truth images and measurement data, respectively. It is important to note here that the training samples are unpaired, i.e., $y_i$ does not necessarily correspond to the noisy measurement of $x_i$, unlike, for instance, a supervised approach such as the learned primal-dual (LPD) method \cite{lpd_tmi}. 
Additionally, we assume that there exists a (potentially regularizing) pseudo-inverse $\ForwardOp^\dagger \colon {Y} \to {X}$ to the forward operator $\ForwardOp$ and define the measure ${\pi}_\dagger \in \PClass_{{X}}$ as ${\pi}_\dagger := \ForwardOp^\dagger_\# (\pi_{\mathrm{data}})$ for $\pi_{\mathrm{data}} \in \PClass_{{Y}}$. 

Then, the idea is to train a regularizer $\RegFunc_{\NNparam}$, parametrized by a neural network, to discriminate between the distributions $\pi_x$ and ${\pi}_\dagger$, the latter representing the distribution of imperfect solutions $\ForwardOp^\dagger y_i$.
More concretely, we compute
\begin{equation}\label{eq:ARGenericTraining}
	\RegFunc_{\widehat{\NNparam}} \colon {X} \to \Real 
	\quad\text{where}\quad
	\widehat{\NNparam} \in \argmin_{\NNparam} L(\NNparam),
\end{equation}
where $L(\NNparam)$ is chosen to be a Wasserstein-flavored loss functional \cite{ar_nips}. In particular, one minimizes 
\begin{equation}
	\label{eq:Wasserstein_Loss_numeric}
	L(\NNparam) := 
	\Expect_{\mathbf{x} \sim \pi_x} \left[\RegFunc_{\NNparam}(\mathbf{x}) \right]
	-
	\Expect_{\mathbf{x} \sim {\pi}_\dagger } \left[\RegFunc_{\NNparam}(\mathbf{x})  \right] 
	+ \lambda\, \Expect_{\mathbf{x}\sim\widetilde{\pi}} \left[ \bigl( \|\nabla \RegFunc_{\NNparam}(\mathbf{x})\| - 1 \bigr)_+^2  \right].
\end{equation}
Here, $\widetilde{\pi}$ denotes the distribution of the random variable $\mathbf{u}=\epsilon\,\mathbf{x}+(1-\epsilon)\mathbf{z}$, where $\mathbf{x}\sim\pi_x$, $\mathbf{z}\sim\pi_{\dagger}$, and $\epsilon$ is drawn uniformly at random from $[0,1]$. The heuristic behind this choice is that a regularizer trained this way will penalize noise and artifacts generated by the pseudo-inverse (and contained in ${\pi}_\dagger$). The term penalizing the gradient norm of $g_\theta$ in \eqref{eq:Wasserstein_Loss_numeric} encourages $g_{\theta}$ to be approximately 1-Lipschitz, which is required for the well-posedness of \eqref{eq:Wasserstein_Loss_numeric} and the stability of the variational solution obtained using the regularizer resulting from \eqref{eq:ARGenericTraining}. When used as a regularizer, it will hence prevent these undesirable features from occurring as a result of adversarial training. The resulting regularizer $\RegFunc_{\widehat{\NNparam}}$ is called an adversarial regularizer (AR). Note that in practical applications, the measures $\pi_x, {\pi}_\dagger \in \PClass_{{X}}$ are replaced with their empirical counterparts given by training data $x_i$ and $\ForwardOp^\dagger y_i$, respectively.

Suppose, one computes a gradient step on the learned regularizer, given by $x_\eta=x-\eta\,\nabla_\mathbf{x}\RegFunc_{\widehat{\NNparam}}(\mathbf{x})$, starting from $x\sim \pi_{\dagger}$. Let ${\pi}^\eta_\dagger$ be the distribution of $x_\eta$. Under appropriate regularity assumptions on the Wasserstein distance $\mathcal{W}({\pi}^\eta_\dagger,\pi_x)$ (see \cite[Theorem 1]{ar_nips}), one can show that 
\[
	\frac{d}{d\eta} \mathcal{W}({\pi}^\eta_\dagger,\pi_x) |_{\eta=0} = - \Expect_{\mathbf{x}\sim{\pi}_\dagger} \|\nabla_\mathbf{x}\RegFunc_{\widehat{\NNparam}}(\mathbf{x})\|^2.
\]
This ensures that by taking a small enough gradient step, one can reduce the Wasserstein distance from the ground truth $\pi_x$. This is a good indicator that using $\RegFunc_{\widehat{\NNparam}}$ as a variational regularization term and consequently penalizing it indeed introduces the highly desirable incentive to align the distribution of regularized solutions with the distribution $\pi_x$ of ground truth samples.
Further, one can show that if the AR is Lipschitz-continuous\footnote{1-Lipschitz continuity is approximately enforced by the gradient penalty term in \eqref{eq:Wasserstein_Loss_numeric}, which does not guarantee, however, that the (AR) is Lipschitz continuous. This property can be enforced by choosing the right network architecture. Indeed, all convolutional neural networks with RELU activations are Lipschitz continuous for some Lipschitz constant $L$, which might be arbitrarily large.}, then a minimizer of the following variational problem exists 
\begin{equation}\label{eq:ARvarcoerc}
	f(y^\delta,\ForwardOp x) + \lambda \left(\RegFunc_{\widehat{\NNparam}}(x) + \epsilon\|x\|_{X}^2\right),
\end{equation}
where the squared norm on $x$ is needed to enforce coercivity. 

Additionally, we can enforce (strong) convexity on $\RegFunc_\NNparam$, leading to the \textit{adversarial convex regularizer} (ACR), to achieve stronger forms of convergence while precluding discontinuities in the reconstruction operator. This necessitates a suitable parameterization of the learned regularizer. One such option is given by input convex neural networks for imposing convexity \cite{amos2017input} on $\RegFunc_{\widehat{\NNparam}}$. 
Given a so-constructed (ACR) $\RegFunc_{\widehat{\NNparam}}$ that is convex in $x$, we then consider a similar regularization functional of the form
\begin{equation}
	\RegFunc(x) = \RegFunc_{\widehat{\NNparam}}(x) + \epsilon \left\|x\right\|_{{X}}^2,
	\label{eq:ACR}
\end{equation}
where $\RegFunc_{\widehat{\NNparam}}:{X}\rightarrow\mathbb{R}$ is the trained (ACR) which we assume to be 1-Lipschitz and convex in $x$. The corresponding variational regularization problem then consists in minimizing
\begin{equation}
	f(y^\delta,\ForwardOp x)+\lambda \RegFunc(x),
	\label{eq:var_loss_maindef1}
\end{equation}
with respect to $x\in{X}$. In this setting, we get the following set of improved theoretical guarantees for the ACR, by following standard arguments in variational calculus for the proofs.
\begin{theorem}[Properties of Adversarial Convex Regularizer \cite{acr_arxiv}] \label{thm:ACR}
	\begin{itemize}
		\item[i.] Existence and uniqueness:
		      The functional in \eqref{eq:var_loss_maindef1} is strongly convex in $x$ and has a unique minimizer $\widehat{x}_{\lambda}\left(y\right)$ for every $y\in {Y}$ and $\lambda>0$.
		      		      		      		      		      		      		      		      		      		      		      		      		      		      		      		      		      		      		      		      		      		      		      		      		          
		\item[ii.] Stability: The optimal solution $\widehat{x}_{\lambda}\left(y\right)$ is continuous in $y$.
		      		      		      		      		      		      		      		      		      		      		      		      		      		      		      		      		      		      		      		      		      		      		      		      		        
		\item[iii.] Convergence: 
		      For $\delta\rightarrow 0$ and $\lambda(\delta) \rightarrow 0$ such that $\displaystyle\frac{\delta}{\lambda(\delta)}\rightarrow 0$, we have that $\widehat{x}_{\lambda}\left(y^{\delta}\right)$ converges to the $\RegFunc$-minimizing solution $x^{\dagger}$ given in \eqref{eq:R_min_def}.
	\end{itemize}
\end{theorem}
Theoretical guarantees notwithstanding, the numerical experiments in \cite{acr_arxiv} (especially, for sparse-view CT reconstruction) indicate a lack of expressive power of ACRs as compared to their nonconvex counterpart AR. This underscores the need to develop techniques that achieve a better compromise between empirical performance and theoretical certificates.    



\subsubsection{The Network Tikhonov (NETT) approach}
Traditionally, regularizers are often chosen as sparsifying transforms with respect to certain features. For instance, total variation (TV) is sparsifying for piecewise constant functions. Similarly, neural networks are often trained in an encoder-decoder (autoencoder) structure, where the encoder is trained to represent the input signal in a low-dimensional space or to find a more efficient, i.e., a sparse structure.
The approach proposed as Network Tikhonov (NETT) in  \cite{nett_paper} follows this paradigm to learn a regularizer. Here, a pretrained network $\Encoder_{\NNparam} \colon {X} \to \EncSpace$ is composed with a regularization functional $\RegFunc \colon \EncSpace \to [0,+\infty]$, such that $\RegFunc \circ \Encoder_{\NNparam} \colon {X} \to [0,+\infty]$ takes small values for desired model parameters and penalizes (by producing larger values for) model parameters with artifacts or other unwanted structures. 
The deep neural network $\Encoder_{\NNparam}$ in this approach is allowed to be a rather general architecture, such as the above-mentioned autoencoder. 
Once trained, the reconstruction is then given as the minimizer of the variational objective
\begin{equation}\label{eq:NETT:func}
	{J}_{\NNparam}(x) := 
	f(y^\delta,\ForwardOp x) + \lambda \RegFunc\bigl( \Encoder_{\NNparam} (x) \bigr).
\end{equation}
Indeed, the NETT approach also provides a provably convergent regularization method under certain analytic conditions on \eqref{eq:NETT:func}, such as weak lower semi-continuity and coercivity of the regularizer $\RegFunc(\Encoder_{\NNparam} (\cdot))$. The primary difference to the ACR in \eqref{eq:ACR}, which achieves convergence in the strong topology of ${X}$ by enforcing convexity, is that the NETT idea achieves convergence in the weak topology of ${X}$. The weak lower semi-continuity and coercivity of $\RegFunc(\Encoder_{\NNparam} (\cdot))$ can be achieved as follows. First, the usual ReLU activation function is replaced by leaky ReLU defined with a small $\tau >0$ as 
\[ 
	{\ell \mathrm{ReLU}}_\tau (s) :=  \max (\tau s, s),
\]
which tends to $- \infty$ for $s \to - \infty$. In combination with the affine linear maps (weight matrices) in $\Encoder_{\NNparam}$, this yields a coercive and weakly lower semi-continuous regularization function 
$\RegFunc \circ \Encoder_{\NNparam}$ for standard choices of $\RegFunc$, such as weighted $\ell_p$-norms $\RegFunc(\xi)= \sum_i v_i |\xi|^p$, with uniformly positive weights $v_i$ and $p \geq 1$.
Finally, we note that strong convergence can be achieved by introducing the novel concept of absolute Bregman distances and imposing stronger conditions on the regularizer.


\section{Regularization by Plug-and-play (PnP) denoising}
\label{sec:Reg_byPnP}
Denoising is the simplest and arguably the most well-studied inverse problem in imaging, with numerous algorithms developed over the past few decades, particularly for removing additive white Gaussian noise from images. It is, therefore, natural to ask if one can leverage off-the-shelf denoisers for solving more complicated image recovery tasks with a non-trivial forward operator. Venkatakrishnan et al. \cite{venkat_pnp_6737048} pioneered the idea of using denoisers within proximal splitting algorithms (e.g., the alternating directions method of multipliers (ADMM) algorithm) in a plug-and-play (PnP) fashion, and the resulting class of algorithms came to be known as the PnP denoising approach. To see the motivation behind using denoisers in place of proximal operators, let us recall the definition of the proximal operator with respect to a (potentially non-smooth) convex functional $g:X\to \mathbb{R}\cup \{+\infty\}$ and a step-size $\tau>0$:   
\begin{align}
	\prox_{\tau\, g}(x) & = \argmin_u \, \frac{1}{2}\|x - u\|^2 + \tau\, g(u). 
	\label{eq:prox_1}
\end{align}
As indicated by \eqref{eq:prox_1}, evaluating the proximal operator amounts to denoising a noisy image $x$ using the Bayesian \textit{maximum a-posteriori probability} (MAP) estimation framework with a Gibbs prior $\propto \exp\left(-\tau\,g(u)\right)$. This denoising interpretation of proximal operators underlies the foundation of PnP approaches, which have been shown to produce excellent reconstruction results for a wide range of imaging inverse problems. A classic and widely popular example of PnP denoising would be to consider it in conjunction with forward-backward splitting (FBS), leading to the following iterative reconstruction algorithm:
\begin{align}
	x_{k+1} = D_{\sigma}\left(x_k - \eta_k\,\nabla f(x_k)\right). 
	\label{eq:pnp_fbs}                             \end{align}
Here, $f$ denotes the data fidelity loss for the underlying inverse problem, $\eta_k>0$ is the step-size at iteration $k$, and $D_{\sigma}$ is a denoiser that eliminates Gaussian noise of standard deviation $\sigma$ from its input. 

Besides the PnP denoising framework within proximal methods, wherein a denoiser implicitly acts as a regularizer, Romano et al. \cite{romano2017RED} proposed an alternative approach to explicitly construct a regularizer as
\begin{align}
	g(x)=\frac{1}{2}x^\top \left(x-D_{\sigma}(x)\right), 
	\label{eq:red_construction}                         \end{align}
while utilizing a denoiser $D_{\sigma}(x)$. One can then seek to minimize the energy functional $f(x)+\lambda\,g(x)$, where $g$ is as defined in \eqref{eq:red_construction}, leading to fixed-point iterative schemes known as the regularization-by-denoising (RED) algorithms. Nevertheless, it was shown subsequently by Schniter et al. \cite{red_schniter} that the \textit{energy minimization} interpretation of the RED algorithms is valid only when (i) the denoiser is \textit{locally homogeneous}, i.e., $D_{\sigma}\left((1+\epsilon)x\right)=(1+\epsilon)D_{\sigma}(x)$ holds for all $x$ with sufficiently small $\epsilon$, and (ii) the Jacobian of $D_{\sigma}$ is symmetric. These conditions are generally not satisfied by generic denoisers, thereby invalidating the energy minimization-based interpretation of RED. Instead, the authors of \cite{red_schniter} developed a new framework called \textit{score-matching} to analyze the convergence of RED algorithms. 



Notwithstanding their empirical success, PnP denoising algorithms such as \eqref{eq:pnp_fbs} does not immediately inherit the convergence properties of the corresponding optimization scheme (in this specific instance, FBS). Studying the convergence of PnP denoising has received a significant amount of attention in the mathematical imaging community in recent years. Arguably, the most natural form of convergence for PnP algorithms of the form \eqref{eq:pnp_fbs} is the stability of the iterations, i.e., to ascertain whether the sequence of iterates $x_k$ generated by a PnP algorithm converges. Such convergence guarantees are typically derived from fixed point theorems, which require showing that the PnP iterations are contractive maps \cite{pnp_admm_chan_2017,pmlr-v97-ryu19a}. For instance, \cite{pmlr-v97-ryu19a} established the fixed-point convergence of PnP-ADMM (i.e., PnP with the \textit{alternating direction method of multipliers} algorithm) under the assumption of Lipschitz continuity of the operator $\left(D_{\sigma}-\id\right)$. The specific result is stated in Theorem \ref{thm:pnp_admm_fp}. 
\begin{theorem}[Fixed-point convergence of PnP-ADMM \cite{pmlr-v97-ryu19a}]\newline
	\label{thm:pnp_admm_fp}
	Consider the PnP-ADMM algorithm, given by
	\begin{align}
		x_{k+\frac{1}{2}} & =\prox_{\tau \, f}\left(z_k\right), x_{k+1} = D_{\sigma}\left(2x_{k+\frac{1}{2}}-z_k\right), \text{\,\,and\,\,}\nonumber \\ z_{k+1}&=z_k+x_{k+1}-x_{k+\frac{1}{2}},
		\label{pnp_drs1}
	\end{align}
	where the data-fidelity loss $f$ is assumed to be $\mu$-strongly convex. One can equivalently express \eqref{pnp_drs1} as the fixed-point iteration $z_{k+1}=\Op{T}(z_k)$, where
	\begin{eqnarray}
		\Op{T}=\frac{1}{2}\id + \frac{1}{2}\left(2D_{\sigma}-\id\right)\left(2\,\prox_{\tau \, f}-\id\right).
		\label{pnp_drs_fp}
	\end{eqnarray}
	Suppose, the denoiser satisfies
	\begin{equation}
		\left\|\left(D_{\sigma}-\id\right)(u)-\left(D_{\sigma}-\id\right)(v)\right\|_2 \leq \epsilon  \left\|u-v\right\|_2, 
		\label{cond_denoiser_pnpDRS}
	\end{equation}
	for all $u,v\in X$ and some $\epsilon>0$, and the strong convexity parameter $\mu$ is such that $\displaystyle\frac{\epsilon}{(1+\epsilon-2\epsilon^2)\,\mu}<\tau$ holds, the operator $\Op{T}$ is contractive and the PnP-ADMM algorithm is fixed-point convergent. That is, $\left(x_k,z_k\right)\rightarrow (x_{\infty},z_{\infty})$, where $(x_{\infty},z_{\infty})$ satisfy 
	\begin{eqnarray}
		x_{\infty}=\prox_{\tau \, f}\left(z_{\infty}\right) \text{\,\,and\,\,} x_{\infty} = D_{\sigma}\left(2x_{\infty}-z_{\infty}\right).
		\label{pnp_drs_final}
	\end{eqnarray}
	As noted in \cite{pmlr-v97-ryu19a}, fixed-point convergence of PnP-ADMM follows from monotone operator theory if $\left(2D_{\sigma}-\id\right)$ is non-expansive, but \eqref{cond_denoiser_pnpDRS} imposes a less restrictive condition on the denoiser.
\end{theorem}
While fixed-point convergence ensures that the PnP iterations are stable, the specific fixed point to which they converge does not automatically minimize a variational energy function. To bridge the gap between classical variational approaches and PnP methods, it is important to derive conditions under which the limit point of PnP iterations can be characterized as the minimizer (or, at least a stationary point) of some regularized variational objective (which, of course, depends on the denoiser). This type of convergence is referred to as \textit{objective convergence} and is stronger than fixed-point convergence. 

Objective convergence of PnP with classical (pseudo) linear denoisers (e.g., non-local means denoiser) has been established in \cite{kunal_9380942}. Hurault et al. \cite{gs_denoiser_hurault_2021} showed that PnP with a denoiser constructed as a gradient field (referred to as gradient-step (GS) denoisers) converges to the stationary point of a (possibly non-convex) variational objective (c.f. Theorem \ref{thm:pnp_admm_gsd}). The construction of GS denoisers is motivated by Tweedie's identity: the optimal minimum mean-squared error (MMSE) Gaussian denoiser is given by
\begin{equation}
	D_{\sigma}^*(x):=\mathbb{E}\left[\mathbf{x}_0|\mathbf{x}=x\right] = x+\sigma^2\,\nabla \log p_{\sigma}(x).
	\label{eq:tweedie}
\end{equation}
Here, $\mathbf{x}=\mathbf{x}_0+\sigma\,\mathbf{z}$, where $\mathbf{x}\sim\mathcal{N}(0,\id)$, is the Gaussian noise (with variance $\sigma^2$) corrupted version of the clean image $\mathbf{x}_0\in X\subseteq \mathbb{R}^d$ and 
\begin{equation}
	p_{\sigma}(x)=\frac{1}{(2\pi\sigma^2)^{\frac{d}{2}}}\int\exp\left(-\frac{\|x-x_0\|_2^2}{2\sigma^2}\right)p(x_0)\,\mathrm{d}x_0.
	\label{eq:smoothed_pdf}
\end{equation}
Indeed, the optimal Gaussian denoiser is of the form $D_{\sigma}^*(x)=x-\nabla\,g^{*}_{\sigma}(x)$, where $g^{*}_{\sigma}$ is the negative log of the smoothed distribution $p_{\sigma}$ defined in \eqref{eq:smoothed_pdf}, which has a structure identical to that of a GS denoiser. 

\begin{theorem}[Objective convergence of PnP iterations with gradient-step (GS) denoisers \cite{gs_denoiser_hurault_2021}]\newline
	\label{thm:pnp_admm_gsd}
	Suppose, the denoiser is constructed as a gradient-step (GS) denoiser, i.e., $D_{\sigma}=\id-\nabla g_{\sigma}$, where $g_{\sigma}$ is proper, lower semi-continuous, and differentiable with an $L$-Lipschitz gradient. The PnP algorithm proposed in \cite{gs_denoiser_hurault_2021} is given by
	\begin{align}
		x_{k+1} & = \prox_{\tau \, f}\left(x_k-\tau \,\lambda \,\nabla g_{\sigma}(x_k)\right)\nonumber \\ &=\prox_{\tau \, f}\circ \left(\tau\lambda\, D_{\sigma}+(1-\tau\lambda\, \id)\right)(x_k),
		\label{eq:pnp_gs_hqs}
	\end{align}
	where $f \colon X\to\Real\cup \{+\infty\}$ denotes data-fidelity and is assumed to be convex and lower semi-continuous. Then, the following guarantees hold for $\tau<\frac{1}{\lambda\, L}$:
	\begin{enumerate}
		\item The sequence $F(x_k)$, where $F=f+\lambda\, g_{\sigma}$, is non-increasing and convergent.
		\item $\left\|x_{k+1}-x_k\right\|_2 \to 0$, which indicates that iterations are stable, in the sense that they do not diverge if one iterates indefinitely. 
		\item All limit points of $\{x_k\}$ are stationary points of $F(x)$. 
	\end{enumerate}
	Notably, the PnP iteration defined by \eqref{eq:pnp_gs_hqs} is exactly equivalent to proximal gradient descent on $f+\lambda\,g_{\sigma}$, with a potentially non-convex $g_{\sigma}$.
\end{theorem}
While objective convergence ensures a one-to-one connection between PnP iterates with the minimization of a variational objective, it does not provide any guarantees about the regularizing properties of the solution that the iterates converge to. In the same spirit as classical regularization theory, it is therefore desirable to be able to control the implicit regularization effected by the denoiser in PnP algorithms and analyze the asymptotic behavior of the PnP reconstruction as the noise level and the regularization strength tend to vanish. More precisely, assuming that the PnP iterations converge to a solution $\hat{x}\left(y^\delta,\sigma,\lambda\right)$, where $\sigma$ is a parameter associated with the denoiser and $\lambda$ is an explicit regularization penalty, one would like to obtain appropriate selection rules for $\sigma$ and/or $\lambda$ such that $\hat{x}\left(y^\delta,\sigma,\lambda\right)$ exhibits convergence akin to \eqref{eq:R_min_def} in the limit as $\delta\rightarrow 0$. To the best of our knowledge, some progress in this direction was first made in \cite{ebner2022plugandplay}, and the precise convergence result is stated in Theorem \ref{thm:pnp_conv_reg_haltmeyer}.
\begin{theorem}[Convergent plug-and-play (PnP) regularization \cite{ebner2022plugandplay}]\newline
	\label{thm:pnp_conv_reg_haltmeyer}
	Consider the PnP-FBS iterates of the form
	\begin{equation}
		x_{\lambda,k+1}^{\delta} = D_{\lambda}\left(x_{\lambda,k}^{\delta}-\eta\,A^*\left(Ax_{\lambda,k}^{\delta}-y^{\delta}\right)\right),
		\label{eq:pnp_conv_reg_thm}
	\end{equation}
	where $D_{\lambda}$ is a denoiser with a tuneable regularization parameter $\lambda$. Let $\PnP\left(\lambda,y^{\delta}\right)$ be the fixed point of the PnP iteration \eqref{eq:pnp_conv_reg_thm}. For any $y\in\range(A)$ and any sequence $\delta_k>0$ of noise levels converging to $0$, there exists a sequence $\lambda_k$ of regularization parameters converging to $0$ such that for all $y_k$ with $\|y_k-y^0\|_2\leq \delta_k$, the following hold under appropriate assumptions on the denoiser (see Definition 3.1 in \cite{ebner2022plugandplay} for details):
	\begin{enumerate}
		\item $\PnP\left(\lambda,y^{\delta}\right)$ is continuous in $y^{\delta}$ for any $\lambda>0$;
		\item The sequence $\left(\PnP\left(\lambda_k,y_k\right)\right)_{k\in\mathbb{N}}$ has a weakly convergent subsequence; and
		\item The limit of every weakly convergent subsequence of $\left(\PnP\left(\lambda_k,y_k\right)\right)_{k\in\mathbb{N}}$ is a solution of the operator equation $y^0=Ax$.
	\end{enumerate}
\end{theorem}
The result in \cite{ebner2022plugandplay}, although the first of its kind, makes fairly restrictive assumptions on the denoiser. In particular, the denoiser needs to be \textit{contractive}, which is not satisfied by most practical denoisers, especially denoisers modeled using deep CNNs. This led us to pose the following question: are PnP approaches with more general and expressive denoisers also convergent regularization methods? This question is perhaps more tractable if one can associate the PnP solution (after convergence) with the minimizer of an underlying variational objective. We, therefore, first consider gradient-step denoisers, for which it is possible to establish such a connection (see Theorem \ref{thm:pnp_admm_gsd}). Treating $\lambda$ in \eqref{eq:pnp_gs_hqs} as an explicit regularization parameter while using a fixed, pre-trained denoiser, one can interpret the converged PnP solution as a minimizer of $f+\lambda\,g_{\sigma}$, where $\lambda$ is varied depending on the noise level $\delta$ in the measurement data and $\sigma$ is kept fixed. The numerical results for image deblurring in Figure \ref{fig:GS_denoiser_deblurring_new} seem to indicate that gradient-step PnP is indeed a convergent regularization scheme, while the classical theory only guarantees stability akin to what is shown in \cite{ar_nips} subject to $g_{\sigma}$ being coercive and bounded below. In addition, the role of $\sigma$ as an implicit regularization parameter is not exploited, and it is kept unchanged regardless of the noise level in the measurement. This, in part, is due to the fact that the behavior of $g_{\sigma}$ w.r.t. $\sigma$ is non-trivial to characterize in a precise manner, leading to difficulties in tuning $\sigma$ based on $\delta$. In order to rigorously establish convergence, together with developing a principled approach to control the regularization strength arising from the denoiser, we consider PnP with linear denoisers in the next section.

\begin{figure}
	\centering
	\begin{subfigure}{0.80\textwidth}
		\includegraphics[width=\textwidth]{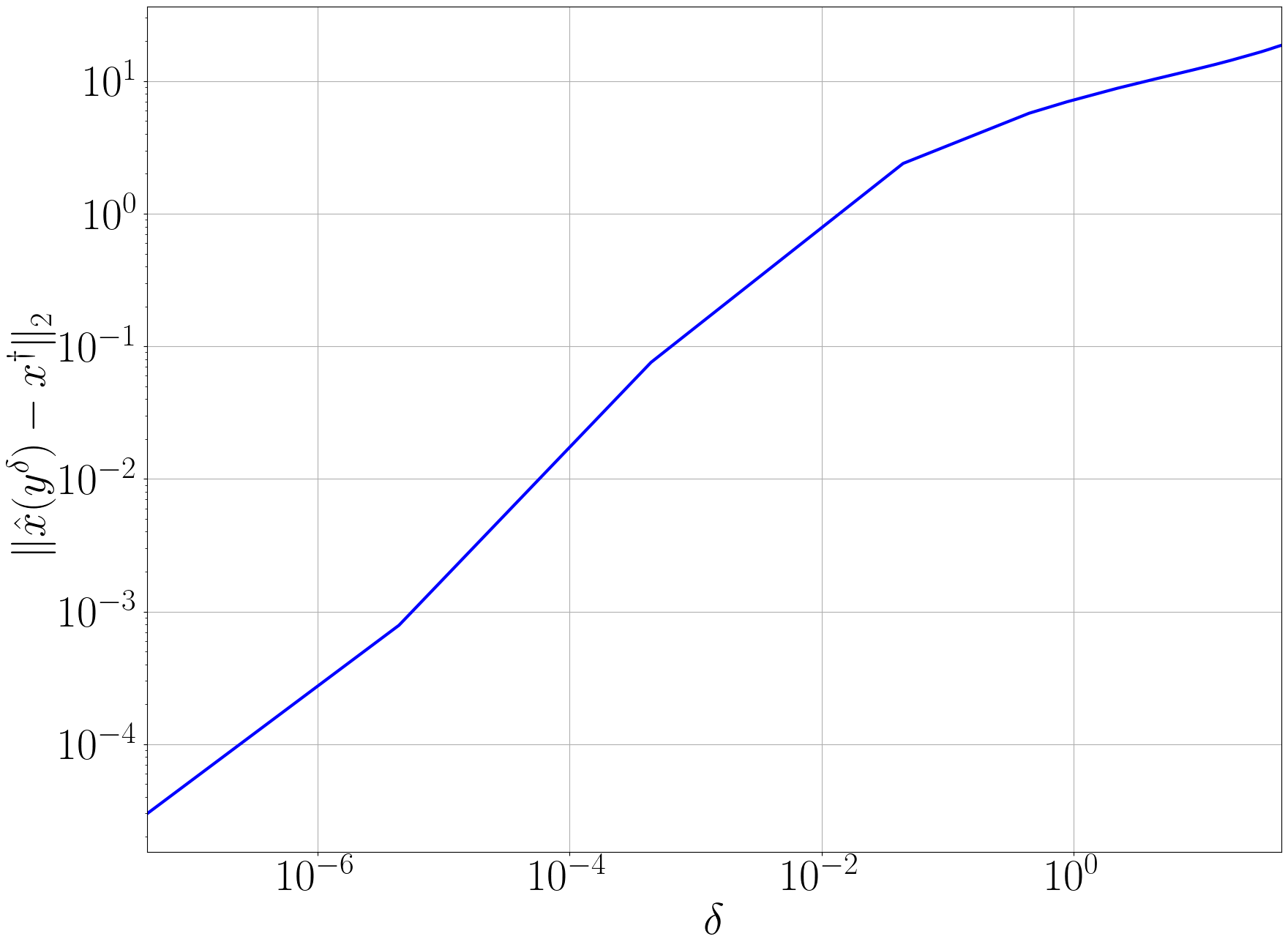}
		\caption{noise level vs. distance from the $g_{\sigma}$-minimizing solution $x^{\dagger}$.}
	\end{subfigure}
	\hfill
	\begin{subfigure}{0.30\textwidth}
		\includegraphics[width=\textwidth]{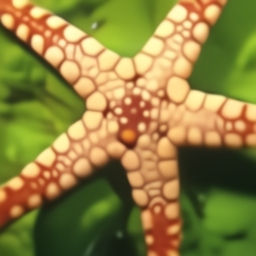}
		\caption{$\hat{x}\left(y^{\sigma_0}\right), \sigma_0=10\%$}
	\end{subfigure}
	\begin{subfigure}{0.30\textwidth}
		\includegraphics[width=\textwidth]{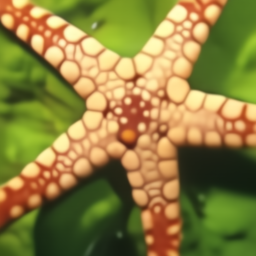}
		\caption{$\hat{x}\left(y^{\sigma_0}\right),\sigma_0=8\%$}
	\end{subfigure}
	\begin{subfigure}{0.30\textwidth}
		\includegraphics[width=\textwidth]{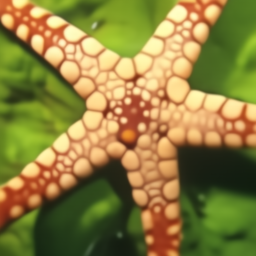}
		\caption{$\hat{x}\left(y^{\sigma_0}\right),\sigma_0=5\%$}
	\end{subfigure}
	\begin{subfigure}{0.30\textwidth}
		\includegraphics[width=\textwidth]{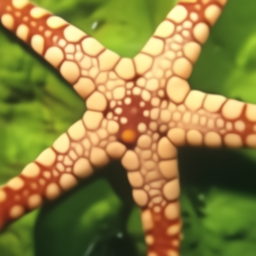}
		\caption{$\hat{x}\left(y^{\sigma_0}\right),\sigma_0=2\%$}
	\end{subfigure}
	\begin{subfigure}{0.30\textwidth}
		\includegraphics[width=\textwidth]{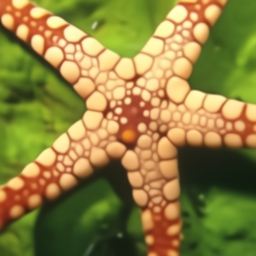}
		\caption{$\hat{x}\left(y^{\sigma_0}\right),\sigma_0=1\%$}
	\end{subfigure}
	\begin{subfigure}{0.30\textwidth}
		\includegraphics[width=\textwidth]{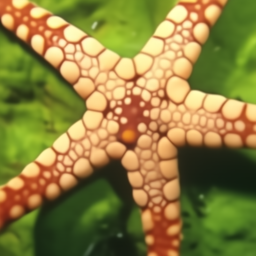}
		\caption{$\hat{x}\left(y^{\sigma_0}\right),\sigma_0=0.5\%$}
	\end{subfigure}
	\begin{subfigure}{0.30\textwidth}
		\includegraphics[width=\textwidth]{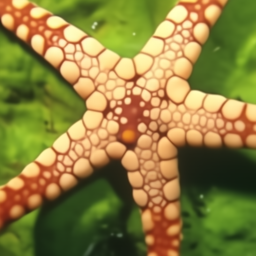}
		\caption{$\hat{x}\left(y^{\sigma_0}\right),\sigma_0=0.1\%$}
	\end{subfigure}
	\begin{subfigure}{0.30\textwidth}
		\includegraphics[width=\textwidth]{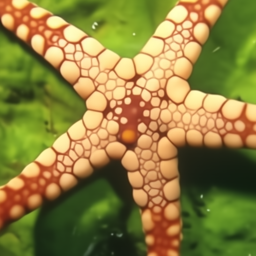}
		\caption{$x^{\dagger}, \sigma_0=0$}
	\end{subfigure}
	\begin{subfigure}{0.30\textwidth}
		\includegraphics[width=\textwidth]{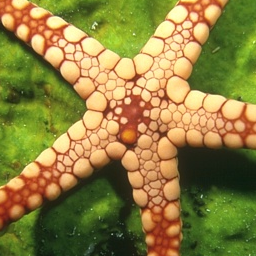}
		\caption{ground-truth}
	\end{subfigure}
																											
	\caption{\small{PnP gradient-step DRUNet denoiser as a convergent regularization method for image deblurring. The PnP scheme for reconstruction minimizes variational energy of the form $f+\lambda\,g_{\sigma}$, where $f$ is the fidelity and $g_{\sigma}$ is the regularizer induced by a pre-trained denoiser. The input blurry image is given by $y^{\sigma_0}=Ax+w$, where $A$ is a Gaussian blur kernel and $w$ is additive Gaussian noise with variance $\sigma_0^2$. The images from (b) to (i) are the deblurred images $\hat{x}\left(y^{\sigma_0}\right)$ corresponding to the noise level $\sigma_0$ (expressed as the \% of maximum pixel value 255.0 in the ground truth). The regularization parameter is selected as $\lambda=c\,\sigma_0+\epsilon$, where the constant $c=0.04$ and $\epsilon=10^{-4}$.}}
	\label{fig:GS_denoiser_deblurring_new}
\end{figure}



\section{Controlling the regularization strength in PnP}
\label{sec:controll_regPnP}
One fundamental question that arises when applying learned denoisers for solving inverse problems using PnP concerns itself with how to adjust the regularization strength that is applied. Indeed, learned denoisers are typically trained at a fixed noise level, whereas their practical application to inverse problems in a PnP framework and the theoretical notion of convergent regularization both require one to have certain control over the regularization strength.

An approach that has been shown to be beneficial in practice is the \textit{denoiser scaling} approach~\cite{xuBoostingPerformancePlugandPlay2020}: given a denoiser $D_\sigma$ (designed for denoising at a given noise level $\sigma$), we introduce an extra scaling parameter $\alpha > 0$, and define the scaled denoisers $\{D_{\sigma, \alpha}\}_{\alpha > 0}$ as 
\begin{equation}
	D_{\sigma, \alpha}(x) = \frac{1}{\alpha} D_\sigma (\alpha x).
	\label{eq:1hom_denoiser_scaling}
\end{equation}
This choice of scaling is motivated by the fact that if $J:X\to \Real\cup \{\infty\}$ is 1-homogeneous (i.e., $J(\tau\,u)=\tau J(u)$, for $\tau>0$) and its proximal operator is well-defined, we have
\begin{align*}
	\prox_{\tau J}(x) & = \argmin_y \frac{1}{2} \|x - y\|^2 + \tau J(y)                                       \\
	                  & = \argmin_y \frac{\tau^2}{2}\Big\|\frac{x}{\tau} - \frac{y}{\tau}\Big\|^2 + \tau J(y) \\
	                  & =\tau\argmin_u \frac{1}{2}\Big\|\frac{x}{\tau} - u\Big\|^2 + J(u)                     \\
	                  & = \tau \prox_J\Big(\frac{x}{\tau}\Big).                                               
\end{align*}
In other words, if $D_\sigma = \prox_J$, then $D_{\sigma, \alpha} = \prox_{J/\alpha}$. Let us note that the choice of this particular scaling, while natural (norms and seminorms are 1-homogeneous, for example), is somewhat arbitrary. Indeed, suppose that $J$ is instead $c$-homogeneous for some $c > 0$, i.e., $J(\delta\,u)=\delta^c J(u)$ for any $u$ and $\delta>0$. We have, with $\delta >0$ arbitrary,
\begin{align*}
	\prox_{\tau J}(x) & = \argmin_y \frac{1}{2}\|x - y\|^2 + \tau J(y)                                                         \\
	                  & =\argmin_y \frac{\delta^2}{2} \Big\|\frac{x}{\delta} - \frac{y}{\delta}\Big\|^2 + \tau J(y)            \\
	                  & =\delta \argmin_u \frac{1}{2} \Big\|\frac{x}{\delta} - u \Big\|^2 + \frac{\tau}{\delta^2} J (\delta u) \\
	                  & =\delta \argmin_u \frac{1}{2} \Big\|\frac{x}{\delta} - u\Big\|^2 + \frac{\tau}{\delta^{2-c}} J(u).     
\end{align*}
Choosing $\delta = \tau^{\frac{1}{2-c}}$, we find that
\begin{equation}
	\prox_{\tau J}(x) = \tau^{\frac{1}{2 - c}} \prox_J(\tau^{\frac{1}{c - 2}} x),
	\label{eq:denoiser_scaling_general}
\end{equation}
which agrees with the result for 1-homogeneous functionals and generalizes it, except for 2-homogeneous functionals where the above derivation does not work. In fact, this leads nicely into a setting where no form of denoiser scaling as in Equation~\eqref{eq:denoiser_scaling_general} can possibly be used to control the regularization strength to give a convergent regularization: for linear denoisers the multiplicative factor inside the denoiser can be pulled out and canceled against the factor outside of it.
\subsection{Controlling the regularization strength of a linear denoiser}
Let us consider the setting in which we have a linear denoiser $D_\sigma :X\to X$. If we are to interpret it as a proximal operator of some underlying functional, we must assume that it is a symmetric, positive semi-definite (p.s.d.) operator, and if we assume that the underlying functional is convex as well, then $D_\sigma$ must be non-expansive in addition. These properties are direct consequences of the characterization of proximal operators given in~\cite{moreauProximiteDualiteDans1965} and generalized (to potentially non-convex functionals) in~\cite{gribonvalCharacterizationProximityOperators2020}. Let us restrict to the case where $D_\sigma$ is non-expansive, bypassing the potential difficulties of non-convexity of the underlying variational problem. In fact, we will assume that $D_\sigma$ is contractive, i.e.\ $\|D_\sigma\|<1$, which as we will see later corresponds to assuming that the underlying regularization functional is coercive. Furthermore, we will assume that $D_\sigma$ is bounded from below, i.e. $\|D_\sigma(x)\| \geq c \|x\|$ for some $c >0$, so that $D_\sigma^{-1}$ exists and is a bounded operator. 
\begin{remark}
	In practice, the assumption of symmetry can be relaxed somewhat by taking a different perspective: in~\cite{gavaskarPlugandPlayRegularizationUsing2021} it is shown in finite dimensions that any denoiser which is similar to a symmetric p.s.d.\ matrix is admissible in PnP applications. Indeed, in this case we can find a modified inner product, with respect to which the denoiser is a proximal operator.
\end{remark}
Let us study the characterization of proximal operators in more detail for the linear denoiser $D_\sigma$. The goal is to understand the underlying functional $J:X\to \Real$ such that $D_\sigma = \prox_J$. Note first that it is immediate from the definition (Equation~\eqref{eq:prox_1}) that we can only hope to recover $J$ up to an additive constant. We have
\[D_\sigma = \prox_J = (\id + \partial J)^{-1} = (\partial [J + \|\cdot \|^2])^{-1},\]
with $\partial$ being the subdifferential. On the other hand, since $D_\sigma$ is linear, it is the gradient of the convex quadratic functional $\phi: X \to \Real$ given by $\phi(x) =\langle x, D_\sigma x\rangle / 2$, i.e.\ $D_\sigma = \nabla \phi$. Using a standard result from convex analysis, we know that $(\nabla \phi )^{-1} = \nabla \phi^*$, where $\phi^*$ is the convex conjugate of $\phi$, and 
\[\nabla \phi^* = \partial[J + \frac{1}{2} \|\cdot \|^2].\]
Hence, up to an irrelevant additive constant, we find that the underlying regularization functional $J$ corresponding to $D_\sigma$ is given by
\begin{equation}
	J(x) = \phi^*(x) - \frac{1}{2}\|x\|^2 = \frac{1}{2} \langle x, (D_\sigma^{-1} - \id) x\rangle.
	\label{eq:linear_prox_underlying}
\end{equation}
The most common way of controlling the regularization strength, when we have access to the underlying regularization functional $J$, is to simply scale it: introduce a parameter $\tau >0$ and consider $\prox_{\tau J}$. If we apply this to Equation~\eqref{eq:linear_prox_underlying}, we obtain
\[\tau J(x) =  \frac{1}{2} \langle x, ([\tau D_\sigma^{-1} - (\tau - 1) \id ] - \id)x\rangle,\]
which suggests, by following the above reasoning in reverse, that 
\begin{equation}
	\prox_{\tau J} = (\tau D_\sigma^{-1} - (\tau - 1)\id)^{-1} = h_\tau(D_\sigma).
	\label{eq:spectral_filtering_denoiser}
\end{equation}
Here $h_\tau : \Real \to \Real$, given by $h_\tau(\lambda) = \lambda/(\tau - \lambda(\tau -  1))$ is applied to $D_\sigma$ using the functional calculus. The takeaway message of the preceding derivation is that we can perform a \emph{spectral filtering} operation on the linear denoiser $D_\sigma$ to control its regularization strength. In fact, more general filter functions $h_\tau$ than the one seen here can be used, as we will see in what follows.
\begin{remark}
	It is worth contrasting the spectral filtering approach proposed here with well-established spectral filtering approaches to regularization of linear, ill-posed, inverse problems~\cite{engl1996regularization}: whereas the traditional approaches operate on the forward operator to enact a regularization effect, we operate on the denoiser (agnostic about the forward operator to which the denoiser will be applied) to control its regularization strength.
\end{remark}
To get a better understanding of what the spectral filtering operation does to a denoiser, consider Figure~\ref{fig:eigenvals_spectral_filtering}. This will help us get an idea of what we should ask of generalized filter functions, i.e.\ filter functions that do not just implement a scaling of the underlying regularization functional.
\begin{figure}[!htb]
	\centering
	\includegraphics[scale=1]{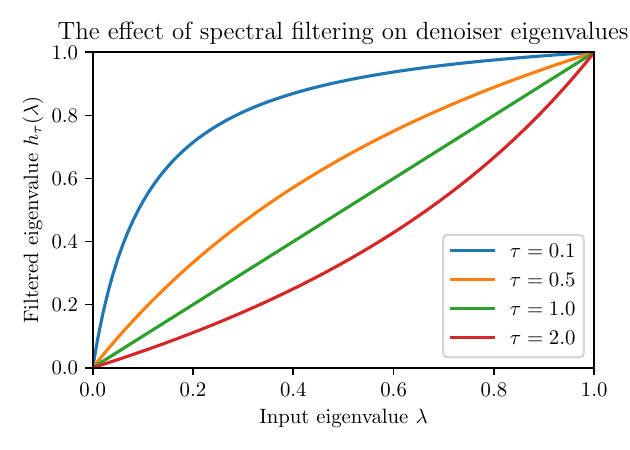}
	\caption{\small{The effect of filtering the denoiser as in~\eqref{eq:spectral_filtering_denoiser}. In accordance with intuition, the spectrum is flattened as $\tau \to 0$: as the regularization strength vanishes, the effect of the denoiser should vanish too.}}
	\label{fig:eigenvals_spectral_filtering}
\end{figure}

\subsection{Convergent regularization through generalized spectral filtering of linear denoisers}

In the previous section, we saw that there is a way in which we can spectrally filter a linear denoiser to effectively scale the underlying regularization functional. Now, we will generalize the conditions on the spectral filter and show that this spectral filtering of linear denoisers allows us to obtain a convergent regularization of linear, ill-posed, inverse problems.

We saw in Equation~\eqref{eq:linear_prox_underlying} that a linear denoiser is related to an underlying regularization functional $J$ as follows: we have $D_\sigma = \prox_J$, where
\[J(x) = \frac{1}{2} \langle x, (D_\sigma^{-1} - \id)x\rangle.\]
Furthermore, we saw the effect of scaling the regularization functional on the corresponding proximal operator. We can generalize this idea and look at 
\[J_\tau(x) := \frac{1}{2}\langle x, (h_\tau(D_\sigma)^{-1} - \id) x\rangle,\]
where $\{h_\tau:\mathbb R\to\mathbb R\}_{\tau \in (0,\infty)}$ is a family of spectral filters that we can apply to $D_\sigma$ using the continuous functional calculus. Let us now derive conditions on the spectral filters $h_\tau$ such that this gives a convergent regularization. For one, since we are assuming that $D_\sigma$ is bounded from below and contractive, we have that $\spec(D_\sigma) \subset (0, 1)$, and the same considerations that led to these assumptions then lead to us asking that $h_\tau(\spec(D_\sigma)) \subset (0, 1)$ for each $\tau > 0$. 

Since we would like to think of $J_\tau$ as somewhat similar to $\tau J$, we ask the question whether the limit
\[J^*(x) := \lim_{\tau \to 0} \frac{1}{\tau}J_\tau(x) = \lim_{\tau \to 0}\frac{1}{2\tau} \langle x, h_\tau(D_\sigma)^{-1} x\rangle - \frac{1}{2\tau} \| x\|^2\]
exists and is sufficiently well-behaved. Indeed, if this limit is well-defined, a natural result to aim for would be that we have convergence to a $J^*$-minimizing least-squares solution to the inverse problem with appropriate choices of $\tau \to 0$ as $\delta\to 0$. 

\begin{remark}
	In Theorem~\ref{theorem:spectral_filtering_linear}, as above, we will assume that the denoiser is contractive, which by~\eqref{eq:linear_prox_underlying} implies that the corresponding regularization functional is coercive. This may be relaxed, by requiring that the kernel of the forward operator is compatible with the denoiser in the sense that the objective function in~\eqref{eq:linear_denoiser_variational_problem} is coercive.
\end{remark}
\begin{theorem}
	\label{theorem:spectral_filtering_linear}
	Suppose that $D_\sigma:X\to X$ is a bounded, linear, self-adjoint operator, which is interpreted as a denoiser. Furthermore, assume that $D_\sigma$ is positive definite, bounded from below, and contractive (so that $\spec(D_\sigma)\subset (0, 1)$). Suppose in addition that we have a bounded, linear forward operator $A: X\to Y$ (assuming w.l.o.g.\ that $\|A\| =1$), and that $\{h_\tau: \mathbb R \to\mathbb R\}_{\tau \in (0, \infty)}$ is a collection of continuous scalar functions satisfying
	\begin{enumerate}[label=\textbf{A.\arabic*}]
		\item \label{assumption:spectrum}\[h_\tau(\spec(D_\sigma)) \subset (0, 1)\,\quad\text{for any } \tau > 0,\]
		\item \label{assumption:convergence}\[r_\tau(\lambda):= \frac{1-h_\tau(\lambda)}{\tau h_\tau(\lambda)} \quad\text{converges uniformly for $\lambda \in \spec(D_\sigma)$ as }\tau \to 0,\]
		      with limit $r^*$ and rate $\|r_\tau - r^*\|_{L^\infty(\spec(D_\sigma))} = o(\tau)$,
		\item 	\label{assumption:coercivity}\[\underline c:=\inf_{\tau > 0, \lambda \in \spec(D_\sigma)} r_\tau(\lambda)> 0, \quad \overline c:=\sup_{\tau > 0, \lambda\in\spec(D_\sigma)} r_\tau(\lambda) < \infty.\]
	\end{enumerate}
	In this setting, let us define (using the continuous functional calculus to apply scalar functions to $D_\sigma$)
	\[J_\tau(x) := \frac{1}{2} \langle x, (h_\tau(D_\sigma)^{-1} -\id) x\rangle=\frac{\tau}{2} \langle x, r_\tau(D_\sigma) x\rangle.\]
	We can compute the solution to the variational problem
	\begin{equation}\hat x = \argmin_{x\in X} \frac{1}{2} \| A x - y\|^2 + J_\tau(x)
		\label{eq:linear_denoiser_variational_problem}
	\end{equation}
	using PnP-FBS:
	\begin{equation}\hat x = \lim_{k\to\infty} x_k, \qquad \text{where}\qquad x_{k+1} = h_\tau(D_\sigma)(x_k - A^*(Ax_k - y)).
		\label{eq:linear_denoiser_pnp_fbs}
	\end{equation}
	By~\ref{assumption:convergence} we can define $J^*(x) = \lim_{\tau \to 0}J_\tau(x)/\tau$. Now, we obtain a convergent regularization when the regularization parameter $\tau^\delta$ is chosen appropriately: suppose that $\tau^\delta\sim\delta$. Assume that we have an underlying image $x^*\in X$, clean measurements $y=Ax^*$, $\{y^\delta\}_{\delta > 0}$ is a sequence in $Y$ satisfying $\|y^\delta - y\|\leq \delta$, and
	\[\hat x(y^\delta,\tau^\delta) = \argmin_{x\in X} \frac{1}{2} \| A x - y^\delta \|^2 + J_{\tau^\delta} (x).\]
	Then $\hat x(y^\delta, \tau^\delta) \to x^\dagger$, where 
	\[x^\dagger = \argmin_{x \in X\,\textrm{s.t.}\,A^*(Ax - y)=0} J^*(x)\]
	is the $J^*$-minimizing least squares solution to the inverse problem $Ax = y$.				
\end{theorem}
																																																																																																																																																																					
\begin{proof} 
	First note that under the assumptions of the theorem, PnP-FBS as described in~\eqref{eq:linear_denoiser_pnp_fbs} is a contractive fixed-point iteration (so that it has a unique fixed point to which it converges) for any $\tau$ and $y$, with fixed points satisfying the optimality condition of the variational problem in~\eqref{eq:linear_denoiser_variational_problem}.
																																																																																																																																																						
	By~\ref{assumption:coercivity}, we can define $\underline J(x) = \inf_\tau J_\tau(x) / \tau $ and $\overline J(x) = \sup_\tau J_\tau(x) / \tau$, so that
	\begin{equation}
		\frac{\underline c}{2} \|x\|^2 \leq \underline J(x) \leq \frac{J_\tau(x)}{\tau} \leq \overline J(x) \leq  \frac{\overline c}{2} \|x\|^2.
		\label{ineq:coercivity}
	\end{equation}
	Taking limits, this also gives us that
	\[\frac{\underline c}{2} \|x\|^2 \leq J^*(x) \leq \frac{\overline c}{2} \|x\|^2.\]
	The above bounds tell us that the $J^*$-minimizing least squares solution to the inverse problem is unique, since it is defined by the minimization of a strongly convex functional on a closed linear subspace of $X$.
																																														
	We have clean measurements $y = Ax^*$, a set of $y^\delta$ such that $\|y - y^\delta \|\leq \delta$ and a parameter choice rule $\delta\mapsto \tau^\delta$ satisfying $\tau^\delta \sim \delta$ as $\tau \to 0$.  We are considering the corresponding set of reconstructions 
	\[\hat x(y^\delta, \tau^\delta) = \argmin_{x\in X} \frac{1}{2} \|A x - y^\delta \|^2 + J_{\tau^\delta}(x).\]
	By the remarks above, we can compute these reconstructions using~\eqref{eq:linear_denoiser_pnp_fbs}. 																																					
	For the sake of the proof, let us also define the variational reconstruction operators with a static regularization functional $J^*$, as follows
	\begin{equation}\hat x_\textrm{static}(y, \tau):= \argmin_{x\in X} \frac{1}{2} \|A x - y \|^2 + \tau J^*(x).
		\label{eq:static_regularization_variational_problem}
	\end{equation}
	This static regularization approach, with the parameter choice that we are using, is a convergent regularization by the existing theory (this is guaranteed, for example, by the general result in~\cite[Proposition 3.32]{scherzer2009variational}): $\hat x_\textrm{static}(y^\delta, \tau^\delta) \to x^\dagger$ with $x^\dagger$ the $J^*$-minimizing least squares solution to the inverse problem. Furthermore, the triangle inequality gives us that
	\begin{equation}\|\hat x(y^\delta, \tau^\delta) - x^\dagger \| \leq \|\hat x(y^\delta, \tau^\delta) - \hat x_\textrm{static}(y^\delta, \tau^\delta)\| + \|x_\textrm{static}(y^\delta, \tau^\delta) - x^\dagger\|,
		\label{ineq:convergence_triangle_inequality}
	\end{equation}
	so it suffices to show that 
	\[\|\hat x(y^\delta, \tau^\delta) - \hat x_\textrm{static}(y^\delta, \tau^\delta)\|\to 0\quad\text{as }\delta \to 0.\]
	We can write 
	\[\hat x(y, \tau) = [A^*A + \tau r_\tau(D_\sigma)]^{-1} A^* y\]
	and
	\[\hat x_\textrm{static}(y, \tau) = [A^*A + \tau r^*(D_\sigma)]^{-1} A^* y,\]
	so we just need to show that $\|M_\tau^{-1} - M_{\tau,\textrm{static}}^{-1}\|\to 0 $ as $\tau \to 0$ (since $\tau^\delta \sim \delta$), where $M_\tau = A^*A + \tau r_\tau(D)$ and $M_{\tau, \textrm{static}} = A^*A + \tau r^*(D)$. We have
	\begin{align}
		M_\tau^{-1} - M_{\tau, \textrm{static}}^{-1} & = [M_{\tau, \textrm{static}} + \tau (r_\tau(D_\sigma) - r^*(D_\sigma))]^{-1} - M_{\tau, \textrm{static}} ^{-1}                \notag                                         \\
		                                             & = \Big[[\id + \tau M_{\tau, \textrm{static}}^{-1} (r_\tau (D_\sigma) - r^*(D_\sigma))]^{-1} -\id\Big] M_{\tau, \textrm{static}}^{-1}. \label{eq:difference_filtering_static} 
	\end{align}
	We will expand the inner matrix inversion using a Neumann series. Note first (by~\ref{assumption:coercivity}) that $M_{\tau, \textrm{static}}$ is bounded from below: $\|M_\tau x\|\geq \underline c \tau \|x\|$. As a result, $\|M_{\tau,\textrm{static}}^{-1}\| \leq 1/(\underline c \tau)$ and we can estimate 
	\begin{equation}\|\tau M_{\tau, \textrm{static}}^{-1} (r_\tau (D_\sigma) - r^*(D_\sigma))\| \leq \tau \frac{1}{\underline c \tau} \|r_\tau(D_\sigma) - r^*(D_\sigma)\| = \frac{\|r_\tau - r^*\|_{L^\infty(\spec(D_\sigma))}}{\underline c}.
		\label{ineq:neumann_bound}
	\end{equation}
	Since \ref{assumption:convergence} tells us that $\|r_\tau - r^*\|_{L^\infty(\spec(D_\sigma))} \to 0$ as $\tau \to 0$, this must be smaller than $1$ for sufficiently small $\tau$, which is a sufficient condition for absolute convergence of the Neumann series. Using this and~\eqref{eq:difference_filtering_static}, we see that
	\begin{align*}
		M_\tau^{-1} - M_{\tau, \textrm{static}}^{-1} & = \Big[\sum\limits_{n=0}^\infty [-\tau M_{\tau,\textrm{static}}^{-1} (r_\tau(D_\sigma) - r^*(D_\sigma)]^n - \id \Big] M_{\tau, \textrm{static}}^{-1}      \\
		                                             & =  \Big[\sum\limits_{n=1}^\infty [-\tau M_{\tau, \textrm{static}}(r_\tau(D_\sigma)-r^*(D_\sigma))]^n\Big]M_{\tau, \textrm{static}} ^{-1}                . 
	\end{align*}
	Finally, we can simply estimate its norm from this as follows, using~\eqref{ineq:neumann_bound} and the fact $\|M_{\tau,\textrm{static}}^{-1} \|\leq 1 / (\underline c \tau)$:
	\begin{align*}
		\|M_\tau^{-1} - M_{\tau, \textrm{static}}^{-1}\| & \leq \sum\limits_{n=1}^\infty \Big(\frac{\|r_\tau - r^*\|_{L^\infty(\spec(D_\sigma))}}{\underline c}\Big)^n \frac{1}{\underline c \tau}                                      \\
		                                                 & =\frac{\|r_\tau - r^*\|_{L^\infty(\spec(D_\sigma))}}{\underline c}\frac{1}{1 - \frac{\|r_\tau - r^*\|_{L^\infty(\spec(D_\sigma))}}{\underline c}} \frac{1}{\underline c\tau} \\
		                                                 & =\frac{1}{\tau}\frac{\|r_\tau - r^*\|_{L^\infty(\spec(D_\sigma))}}{\underline c^2 - \underline c\|r_\tau - r^*\|_{L^\infty(\spec(D_\sigma))}} .                              
	\end{align*}
	Since we have assumed that $\|r_\tau - r^* \|_{L^\infty(\spec(D_\sigma))} = o(\tau)$ as $\tau \to 0$, we find by the above reasoning that $\|\hat x(y^\delta, \tau^\delta) - \hat x_\textrm{static}(y^\delta, \tau^\delta)\| \to 0$. Recalling the inequality in \eqref{ineq:convergence_triangle_inequality} lets us conclude that the spectral filtering approach is a convergent regularization.
\end{proof}
																																																																																																																																																																					
\begin{example}
	Consider the case previously considered in~\eqref{eq:linear_prox_underlying} and~\eqref{eq:spectral_filtering_denoiser}, corresponding to $h_\tau(\lambda) =\lambda / (\tau (1 - \lambda) + \lambda)$. We have 
	\[r_\tau(\lambda) = \frac{1-h_\tau(\lambda)}{\tau h_\tau(\lambda)} = \frac{1 - \lambda}{\lambda}.\]
	In particular, the assumptions \ref{assumption:coercivity}, \ref{assumption:convergence} and \ref{assumption:spectrum} are trivially satisfied: we have $h_\tau(\lambda) < \lambda$ for $\lambda > 0$, $r^* = r_\tau$ for all $\tau > 0$ and 
	\[	\inf_{\tau > 0, \lambda\in \spec(D_\sigma)} r_\tau(\lambda)= \frac{1 - \lambda_{\max}(D_\sigma)}{\lambda_{\max}(D_\sigma)} > 0\]
	and
	\[
		\quad \sup_{\tau > 0, \lambda\in \spec(D_\sigma)} r_\tau(\lambda)= \frac{1 - \lambda_{\min}(D_\sigma)}{\lambda_{\min}(D_\sigma)} < \infty.\]
		This should come as no surprise, since by the previous discussion, this choice of spectral filtering simply corresponds to the static regularization approach used in the proof of Theorem~\ref{theorem:spectral_filtering_linear}, for which classical theory establishes its convergence properties.
		\end{example}
																																																						
		\subsection{Experiments}
																																																																																																																																																																																											
		In this section, we will demonstrate the use of the spectral filtering approach to control the regularization strength of a learned denoiser, when applied to an inverse problem using PnP-FBS, showing that it in fact gives rise to a practically convergent regularization method. Since the spectral filtering approach was developed for linear denoisers, we first need to decide on a reasonable design for a linear learnable denoiser.
																																																																																																																																																																																											
		In this work, we will modify the U-net architecture~\cite{ronnebergerunet2015}, which continues to be used with great success in image-to-image tasks, and combines a downscaling and upscaling path (as in an autoencoder) with skip connections that connect the corresponding scales before and after the bottleneck. The key insight for us is that the U-net architecture is symmetric, in the following sense: if the downscaling and upscaling operations are linear and each other's transposes, and the activation functions and biases are omitted, the U-net is linear and its transpose is a U-net of the same shape (which can be thought of as running the original U-net in reverse). In particular, it is straightforward to see that we can obtain a symmetric linear U-net in this way by tying weights between the downscaling and upscaling paths. Alternatively, and perhaps more simply, we can take the average of a linear U-net and its transpose to get a symmetric linear denoiser. This is the approach that we will take in the experiments considered in this section, since we can leverage the power of JAX~\cite{jax2018github} to do so: given a linear U-net, we can efficiently compute its vector-Jacobian products to get its transpose. Figure~\ref{fig:compare_linear_nonlinear_unet} shows a comparison of the denoising performance (in the same setting as the one we will consider for the application to inverse problems below) of such a linear U-net with a comparable non-linear U-net. By this, we mean that the networks have the same sizes and the same number of trainable parameters. While the non-linear U-net allows for better reconstructions, most notably in terms of sharpness, both denoisers remove a significant part of the noise in the noisy images. In what follows, we will use the linear U-net $D_{\sigma, \textrm{l}}$ and simply call it $D_\sigma$.
																																																														
		\begin{figure}[!htb]
			\centering
			\includegraphics*[scale=1]{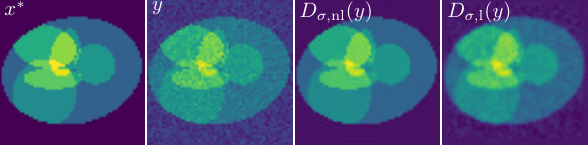}
			\caption{\small{Comparing the denoising performance of a non-linear U-net ($D_{\sigma,\textrm{nl}}$) with a linear, symmetric U-net ($D_{\sigma,\textrm{l}}$) based on the same architecture. Here $x^*$ is a ground truth image and $y$ is the same image, corrupted by Gaussian noise. These images are generated in the same way as the training data was generated. In contrast to $D_{\sigma,\textrm{nl}}$, $D_{\sigma, \textrm{l}}$ struggles to reconstruct sharp edges as it does not contain any non-linearity.  On the other hand, both denoisers significantly improve the signal-to-noise ratio: $y$ has a PSNR of 24.3 dB, $D_{\sigma, \textrm{nl}}(y)$ has a PSNR of 34.0 dB and $D_{\sigma, \textrm{l}}(y)$ has a PSNR of 27.8 dB.}}
			\label{fig:compare_linear_nonlinear_unet}								
		\end{figure}
																																																																
		The experiment that we will consider is concerned with the inverse problem of image reconstruction in computed tomography (CT). We will consider images of size $64\times 64$, consisting of randomly generated ellipse phantoms as in Figure~\ref{fig:spectral_filtering_images}, and simulate CT measurements (sinograms) using the ASTRA toolbox~\cite{van_aarle_astra_2015, aarle_fast_2016} with a parallel beam geometry with 150 equispaced views. A linear U-net is trained as a denoiser on ellipse phantoms corrupted with Gaussian white noise, after which we apply the denoiser in a PnP-FBS manner: denoting the forward operator, which maps images $u$ to (clean) sinograms $y$ by $A$, the noisy measurements $y^\delta$ and the trained denoiser by $D_\sigma$, we iterate
		\[x_{k + 1} = h_{\tau}(D_\sigma)( x_k - \eta A^*(Ax_k - y^\delta)), \qquad \hat x(y^\delta, \tau) = \lim_{k\to\infty} x_k\]
		where $\eta$ is a step size, satisfying $\eta \leq 2/\|A\|^2$, so that the limit is well-defined. Here, we simply use the spectral filters $h_\tau$ corresponding to scaling the underlying regularization functional as seen in Equation~\eqref{eq:spectral_filtering_denoiser}.
																																																																								
		We will consider a sequence of noisy measurements $\{y^\delta\}_{\delta > 0}$ such that $\|y^\delta - y\|\leq \delta$ and a corresponding step size $\tau^\delta \propto \delta$, satisfying the conditions of Theorem~\ref{theorem:spectral_filtering_linear}. In Figure~\ref{fig:spectral_filtering_convergence} and Figure~\ref{fig:spectral_filtering_images} we show that the spectral filtering approach indeed leads to a practically convergent regularization, as predicted by Theorem~\ref{theorem:spectral_filtering_linear}.

		\begin{figure}[!htb]
			\centering
			\includegraphics[scale=1]{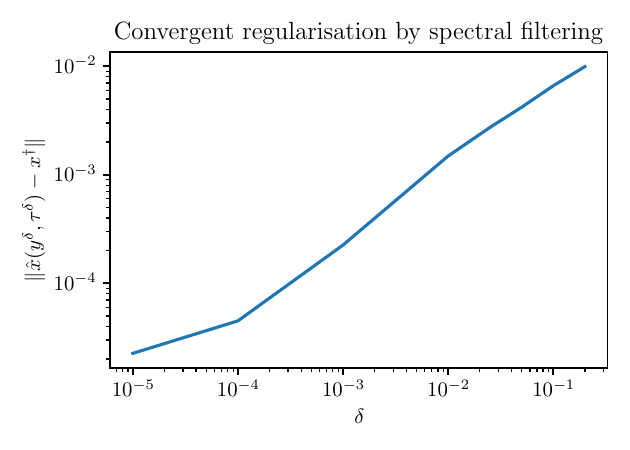}
			\caption{\small{Applying the spectral filtering approach, we observe convergent regularization in practice. Here $x^\dagger$ is the $J^*$-minimizing solution to the noiseless least-squares problem as in Theorem~\ref{theorem:spectral_filtering_linear}.}}
			\label{fig:spectral_filtering_convergence}
		\end{figure}

		\begin{figure}[!htb]
			\centering
			\includegraphics[scale=1]{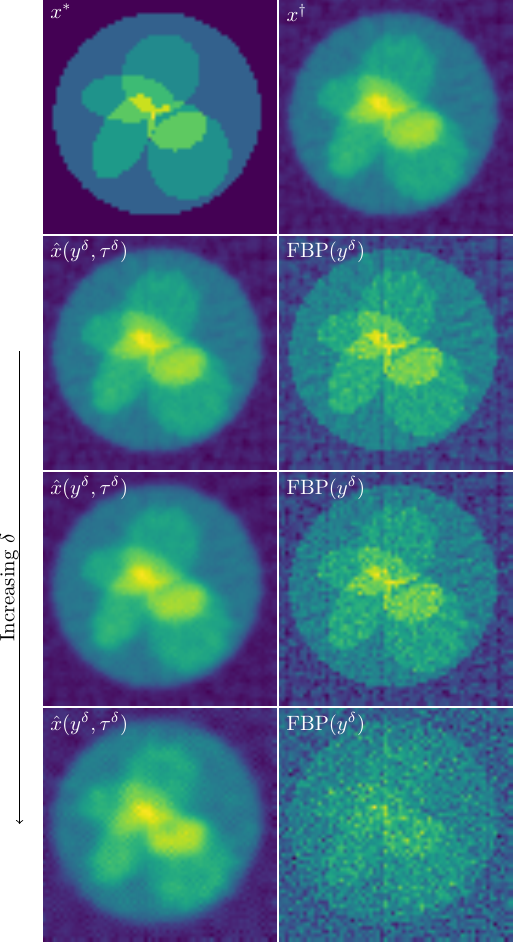}
			\caption{\small{Applying the spectral filtering approach, we observe convergent regularization in practice. We show a selection of snapshots corresponding to the plot in Figure~\ref{fig:spectral_filtering_convergence}. Note that $x^*$ (the underlying ground truth) is distinct from $x^\dagger$ since the forward operator has a non-trivial kernel.}}
			\label{fig:spectral_filtering_images}
		\end{figure}

\section{Conclusions}
The question if a reconstruction algorithm provides a convergent regularization has been long studied in inverse problems, as it provides more than just the knowledge that a solution can be computed at a certain noise level. It tells us that stable solutions exist for all noise realizations and even more importantly that in the limit case, when noise vanishes, we obtain a solution of the underlying operator equation. In other words, we can guarantee mathematically that obtained solutions are indeed solutions to the inverse problem. 

This is in contrast to some novel data-driven approaches where we may only guarantee that obtained solutions are minimizers of the empirical loss, given suitable training data. Consequently, the concept of convergent data-driven reconstructions has gained considerable interest very recently, see for instance~\cite{mukherjee2023learned}. Here, PnP approaches take a special role due to their straightforward connection to convex optimization~\cite{kamilov2023plug} and the possibility to incorporate learned denoisers given by non-linear neural networks. But despite considerable advances in establishing convergence notions, i.e., fixed-point and objective convergence, the question of convergent regularization is still open for general non-linear denoisers. 

In this work, we presented a step forward for learned linear denoisers using the novel concept of spectral filtering of the denoiser. The presented approach allows to establish a provably convergent regularization in the PnP framework. Additionally, this convergence is demonstrated numerically on the inverse problem of CT image reconstruction. As established in Theorem~\ref{theorem:spectral_filtering_linear}, there is some freedom in the choice of filters to apply to the denoiser. In future work, this choice could be studied in more detail. In this direction, it is of particular interest to choose spectral filters that are not too computationally costly to evaluate but still give a way to tune the regularization strength of the denoiser. Indeed, in the present implementation of the method, after training, the denoiser is instantiated as a matrix, the eigen-decomposition of which is computed to apply the spectral filtering. By considering spectral filters given by polynomials, for example, we would circumvent the need to instantiate the denoiser as a matrix and compute a full eigen-decomposition. Besides this, it would be of great interest to study whether there is any reasonable generalization of the denoiser filtering approach to the setting in which the denoiser is non-linear.

In fact, we have observed similar convergence behavior numerically even when using a non-linear denoiser in the PnP gradient-step framework (see Figure~\ref{fig:GS_denoiser_deblurring_new}), suggesting a promising direction for proving that PnP with realistic assumptions on the denoiser can give rise to convergent regularization. The gradient-step framework is, however, just one way of controlling the regularization strength of the learned denoiser. In particular, it relies on flipping the usual splitting of the variational objective and, as a result, requires repeated evaluation of the proximal operator of the data term. This may be very computationally costly if the forward operator is expensive to evaluate. As a result, it is still of great interest to study other ways of controlling the regularization strength of a realistic learned denoiser in PnP that will result in provably convergent regularization.

\section*{Acknowledgments}
CBS acknowledges support from the Philip Leverhulme Prize, the Royal Society Wolfson Fellowship, the EPSRC advanced career fellowship EP/V029428/1, EPSRC grants EP/S026045/1 and EP/T003553/1, EP/N014588/1, EP/T017961/1, the Wellcome Innovator Awards 215733/Z/19/Z and 221633/Z/20/Z, the European Union Horizon 2020 research and innovation programme under the Marie Sk\l odowska-Curie grant agreement No.\ 777826 NoMADS, the Cantab Capital Institute for the Mathematics of Information and the Alan Turing Institute. FS acknowledges support from the EPSRC advanced career fellowship EP/V029428/1. AH acknowledges support from Academy of Finland's Centre of Excellence of Inverse Modelling and Imaging proj. 353093 and the Academy Research Fellow proj. 338408.

\bibliographystyle{spmpsci}      

\bibliography{ref}



\end{document}